\documentclass[11pt]{article}

\usepackage{bbm}
\usepackage[english]{babel}
\usepackage{amsmath}
\usepackage{amsthm}
\usepackage{amssymb}
\usepackage[dvips]{color}

\newtheorem{theorem}{Theorem}[section]
\newtheorem{definition}[theorem]{Definition}
\newtheorem{lemma}[theorem]{Lemma}
\newtheorem{proposition}[theorem]{Proposition}
\newtheorem{corol}[theorem]{Corollary}
\newtheorem{remark}[theorem]{Remark}

\newcommand{\rr}{\mathbb{R}}

\newcommand{\hh}{\mathbb{H}}

\newcommand{\pp}{\partial}

\newcommand{\unx}{\underline{x}}

\title{\bf  Formulations of the  F-functional calculus \\
and some consequences}

\author{
Fabrizio Colombo\\Politecnico di
Milano\\Dipartimento di Matematica\\Via E. Bonardi, 9\\20133 Milano,
Italy\\fabrizio.colombo@polimi.it
\and
Jonathan Gantner  \\
Vienna University of Technology\\
Institute for Analysis \\ and Scientific Computing\\
Wiedner Hauptstra\ss e 8 - 10\\
1040 Wien, Austria\\
jonathan.gantner@gmx.at
 }

\date{ }
\begin{document}

\maketitle

\begin{abstract}
In this paper we introduce the two possible formulations of the F-functional calculus which are based on the Fueter-Sce mapping theorem in integral form and we  introduce the pseudo F-resolvent equation. In the case of dimension $3$
we prove the F-resolvent equation
and  we study the analogue of the Riesz projectors associated with this calculus.
The case of dimension 3 is also useful to  study  the quaternionic version of the F-functional calculus.

\end{abstract}

\noindent {\em Key words}: Fueter-Sce mapping theorem in integral form, the F-spectrum, formulations  of the $F$-functional calculus for $n$-tuples of operators, projectors, quaternionic version of the $F$-functional calculus, quaternionic F-resolvent equation.

\section{Introduction}

Probably the most important functional calculus  for linear operators acting on a Banach space
is the Riesz-Dunford functional calculus, see \cite{ds}.
For the case of $n$-tuples of operators we quote the paper by Anderson \cite{anderson} who developed the Weyl functional calculus, and the work of J.L Taylor and M.E. Taylor  \cite{Taylors, Taylor, Taylor3, TaylorME} who defined a functional calculus for $n$-tuples of operators
using the theory of holomorphic functions of several variables.
Since then, the literature on this topic has been developed in different directions also using different notions of hyperholomorphicity.

In the past few years the theory of slice hyperholomorphic functions, see \cite{Global,CSTrend,DUKE,slicecss,ISRAEL2},
 has been the underline function theory on which  two new functional calculi for several operators and for quaternionic operators have been developed.
 These calculi are the $S$-functional calculus, \cite{acgs, NOTIONCONV, DUKE, SlaisComm, functionalcss} and the $F$-functional calculus, see \cite{NOTIONCONV, F_FUNC_UNB, CoSaSo}.
The first one is based on the Cauchy formula of the slice hyperholomorphic functions, see \cite{CSTrend, DUKE},
 and it applies to $n$-tuples of not necessarily commuting operators and to quaternionic operators.
In this last case the $S$-functional calculus is called quaternionic functional calculus.
 For an overview on slice hyperholomorphic functions and for the $S$-functional calculus  see the monograph \cite{bookfunctional}.
In the recent paper \cite{acgs}  it has been shown that, in order to have a full description of the
$S$-functional calculus we are in need of both the two formulations of this calculus,
 precisely the formulation based on the right $S$-resolvent operator but also the one based on the left $S$-resolvent operator. This is necessary because the $S$-resolvent equation involves both resolvent operators.
 This  equation is a key tool to study the Riesz projectors associated with the $S$-functional calculus.
 We point out that there exists a continuous version of the $S$-functional calculus that can be found in the paper \cite{GMP}.
  The aim of this work is to introduce the formulations of the $F$-functional calculus and
  to  extend some of the results proved in \cite{acgs} to the $F$-functional calculus. This calculus is based on the integral version of the Fueter-Sce mapping theorem.

The Fueter mapping theorem is one of the deepest results in hypercomplex analysis, see \cite{fueter 1}, it gives a  procedure to generate Cauchy-Fueter regular functions starting from holomorphic functions of a complex variable.
In the case of  Clifford algebra valued functions, see \cite{bds,csss,DSS,ghs},  the proof of the analogue of the Fueter mapping theorem  is due to Sce \cite{sce} for $n$ odd and to Qian \cite{qian} for the general case.
Later on, Fueter's theorem has been generalized to the  case in which a function $f$ is multiplied by a monogenic homogeneous polynomial of degree $k$, see \cite{kqs,penapena,pqs,sommen1} and to the case in which the function $f$ is defined on an open set $U$ not necessarily chosen in the upper complex plane, see \cite{qian,Qian3,qian2}.

More recently the problem of the inversion of Fueter-Sce mapping theorem  as been investigated in a series of papers, see \cite{INVER, INVERSEPK, biaxially}.

The $F$-functional calculus is based on the Fueter-Sce mapping theorem in integral form,
this is an integral transform obtained applying suitable powers of the Laplace operator to the
Cauchy kernel of slice hyperholomorphic functions.

The  case of left slice monogenic functions  has already been studied in the paper \cite{CoSaSo}, where the $F$-functional calculus has been introduced.

We observe that to study the Riesz projectors, in the classical case, we are in need of the resolvent equation
\begin{equation}
(\lambda I-A)^{-1}(\mu I-A)^{-1}=\frac{(\lambda I-A)^{-1}-(\mu I-A)^{-1}}{\mu-\lambda},\ \ \ \lambda, \ \mu\in \mathbb{C}\setminus \sigma(A),
\end{equation}
where $A$ is a complex operator on a complex Banach space.

In fact, to study the classical Riesz projectors, we use the fact that the product of the resolvent operators
$(\lambda I-A)^{-1}(\mu I-A)^{-1}$ can be written in terms of the difference $(\lambda I-A)^{-1}-(\mu I-A)^{-1}$ multiplied by the Cauchy kernel of holomorphic functions. As a consequence of the holomorphicity on can prove that
$$
P_\Omega=\int_{\partial \Omega}(\lambda I-A)^{-1}d\lambda,
$$
where $\Omega$ contains part of the spectrum of $A$, is a projector, that is $P^2_\Omega=P_\Omega$.

In the case of the $S$-functional calculus (and in particular its commutative version the $SC$-functional calculus) we can follow the same strategy but in this case the $SC$-resolvent equation
contains both the $SC$-resolvent operators as it has been recently shown in \cite{acgs}.

Precisely, we denote by $T$ the paravector operator defined by $T=T_0+e_1T_1+...+e_nT_n$ where $T_j$, $j=0,1,...,n$, are real bounded operators,
 commuting among themselves, acting on a real Banach space and $e_j$, $j=1,...,n$, are the units of the real Clifford algebra $\mathbb{R}_n$.
The $F$-spectrum of $T$ as is defined as
$$
\sigma_{F}(T)=\{ s\in \mathbb{R}^{n+1}\ \ :\ \ \ \ \ s^2\mathcal{I}-s(T+\overline{T})+T\overline{T}\ \ \
{\rm is\ not\  invertible}\},
$$
where we have set $\overline{T}:=T_0-e_1T_1-...-e_nT_n$. The left $SC$-resolvent operator is defined as
\begin{equation}
S_{C,L}^{-1}(s,T):=(s\mathcal{I}- \overline{T})(s^2\mathcal{I}-s(T+\overline{T})+T\overline{T})^{-1},\ \ \ s\not\in\sigma_{F}(T)
\end{equation}
and  the right $SC$-resolvent operator is
\begin{equation}
S_{C,R}^{-1}(s,T):=(s^2\mathcal{I}-s(T+\overline{T})+T\overline{T})^{-1}(s\mathcal{I}- \overline{T}),\ \ \ s\not\in\sigma_{F}(T).
\end{equation}
In this case, for $s$, $p\not\in \sigma_{F}(T)$,  the $SC$-resolvent equation is:
\begin{equation}
\begin{split}
S_{C,R}^{-1}(s,T)&S_{C,L}^{-1}(p,T)=((S_{C,R}^{-1}(s,T)-S_{C}^{-1}(p,T))p
\\
&
-\overline{s}(S_{C,R}^{-1}(s,T)-S_{C}^{-1}(p,T)))(p^2-2s_0p+|s|^2)^{-1},
\end{split}
\end{equation}
where $p$ and $s$ are paravectors, that is $s=s_0+s_1e_1+...+s_ne_n$ and $|s|^2=s_0^2+s_1^2+...+s_n^2$.
This equation is the key tool to study the Riesz projectors for the $S$-functional calculus and in particular for its commutative version the $SC$-functional calculus.
In the case of the $F$-functional calculus,  let  $n$ be an odd number,
 we define the left $F$-resolvent operator as
\begin{equation}
F_n^L(s,T):=\gamma_n(s\mathcal{I}-\overline{ T})(s^2\mathcal{I}-s(T+\overline{T})+T\overline{T})^{-\frac{n+1}{2}}
,\ \ \ s\not\in\sigma_{F}(T)
\end{equation}
and the right $F$-resolvent operator as
\begin{equation}
F_n^R(s,T):=\gamma_n(s^2\mathcal{I}-s(T+\overline{T})+T\overline{T})^{-\frac{n+1}{2}}(s\mathcal{I}-\overline{ T}),\ \ \ s\not\in\sigma_{F}(T),
\end{equation}
where $\gamma_n$ are suitable constants.
In this case we cannot expect a $F$-resolvent equation as in the case of the Riesz-Dunford functional calculus or  as in the case of the SC-functional calculus because the $F$-functional calculus in based on an integral transform and not on a Cauchy formula.

At the moment we are able to show that, at lest in the case $n=3$, there is a $F$-resolvent equation but it contains,
besides the two $F$-resolvent operators, also the two $SC$-resolvent  operators.
Precisely, for $s$, $p\not\in \sigma_{F}(T)$, it is
\[
\begin{split}
 &F_3^R(s,T)S_{C,L}^{-1}(p,T)+S_{C,R}^{-1}(s,T)F_3^L(p,T)+\gamma_3^{-1}\Big( sF_3^R(s,T)F_3^L(p,T)p
 \\
& -sF_3^R(s,T)TF_3^L(p,T)-F_3^R(s,T)TF_3^L(p,T)p+F_3^R(s,T)T^2F_3^L(p,T)\Big)
 \\
 &
 = \left[ \left( F_3^R(s,T) - F_3^L(p,T)\right)p-\bar{s}\left( F_3^R(s,T) - F_3^L(p,T)\right) \right](p^2 - 2s_0p + |s|^2)^{-1}.
\end{split}
\]
Even though it is more complicated than the $SC$-resolvent equation it is the correct tool to study the analogue of the Riesz projectors for the $F$-functional calculus.
Moreover, the case $n=3$ allows the study in all the details the quaternionic version of the $F$-functional calculus.

We conclude by observing  that important applications of the quaternionic functional calculus and of its $S$-resolvent operators
 can be found in Schur analysis  in the slice hyperholomorphic setting,  see \cite{ acs1, acs2, acs3},  and see \cite{MR2002b:47144} for the classical case. This a very active field that has started a few years ago.

\section{Preliminaries on slice monogenic functions}

In this section we recall some results on slice monogenic functions which will be useful in the sequel,
we refer the reader to \cite{bookfunctional} for more details.

The setting in which we will work is the real Clifford algebra $\rr_n$ over $n$ imaginary units
$e_1,\dots,e_n$ satisfying the relations $e_ie_j+e_je_i=-2\delta_{ij}$. An element in the
Clifford algebra will be denoted by $\sum_A e_Ax_A$ where
$A=i_1\ldots i_r$, $i_\ell\in \{1,2,\ldots, n\}$, $i_1<\ldots <i_r$, is a multi-index,
$e_A=e_{i_1} e_{i_2}\ldots e_{i_r}$ and $e_\emptyset =1$. When $n=1$, we have that $\rr_1$ is the algebra of complex numbers $\mathbb{C}$ (the
only case in which the Clifford algebra is commutative),
while when $n=2$ we obtain the division algebra of real quaternions $\mathbb{H}$. As it is well known,
for $n>2$, the Clifford algebras $\rr_n$ have zero divisors.

In the Clifford algebra $\rr_n$, we can identify some specific elements with the vectors in the Euclidean space $\rr^n$:
an element $(x_1,x_2,\ldots,x_n)\in\rr^n$ can be identified
with a so called 1-vector in the Clifford algebra through the map $(x_1,x_2,\ldots,x_n)\mapsto \unx=
x_1e_1+\ldots+x_ne_n$.
\par\noindent
An element $(x_0,x_1,\ldots,x_n)\in \rr^{n+1}$  will be identified with the element
 $
 x=x_0+\unx=x_0+ \sum_{j=1}^nx_je_j
 $
called, in short, paravector.
The norm of $x\in\rr^{n+1}$ is defined as $|x|^2=x_0^2+x_1^2+\ldots +x_n^2$. The real part  $x_0$ of $x$ will be also denoted by ${\rm Re} (x)$.
A function $f:\ U\subseteq \rr^{n+1}\to\rr_n$ is seen as
a function $f(x)$ of $x$ (and similarly for a function $f(\unx)$ of $\unx\in U\subset \rr^n$).
We will denote by $\mathbb{S}$ the sphere of unit 1-vectors in $\mathbb{R}^n$, i.e.
$$
\mathbb{S}=\{ \unx=e_1x_1+\ldots +e_nx_n\ :\  x_1^2+\ldots +x_n^2=1\}.
$$

\begin{definition}
\label{defsmon}
Let $U\subseteq\mathbb{R}^{n+1}$ be an open set and let
$f: U\to\mathbb{R}_n$ be a real differentiable function. Let
$I\in\mathbb{S}$ and let $f_I$ be the restriction of $f$ to the
complex plane $\mathbb{C}_I$ and denote by $u+Iv$ an element on $\mathbb{C}_I$.
We say that $f$ is a left slice monogenic (for short {\em  s-monogenic}) function if, for every
$I\in\mathbb{S}$, we have
$$
\frac{1}{2}\left(\frac{\partial }{\partial u}f_I(u+Iv)+I\frac{\partial
}{\partial v}f_I(u+Iv)\right)=0
$$
on $U\cap \mathbb{C}_I$.
We will denote by $\mathcal{SM}(U)$ the set of left slice monogenic functions on the open set $U$ or by $\mathcal{SM}^L(U)$ when confusion may arise.
We say that $f$ is a right slice monogenic (for short {\em  right s-monogenic}) function if,
for every
$I\in\mathbb{S}$, we have
$$
\frac{1}{2}\left(\frac{\partial }{\partial u}f_I(u+Iv)+\frac{\partial
}{\partial v}f_I(u+Iv)I\right)=0,
$$
on $U\cap \mathbb{C}_I$.
We will denote by $\mathcal{SM}^R(U)$ the set of right slice monogenic functions on the open set $U$.
\end{definition}

\begin{definition}\label{sphere}
Given an element $x\in\rr^{n+1}$, we define
$$
[x]=\{y\in\rr^{n+1}\ :\ y={\rm Re}(x)+I |\underline{x}|, \ \ I\in \mathbb{S}\}.
$$
\end{definition}

 The set $[x]$ is a $(n-1)$-dimensional sphere in $\rr^{n+1}$. When $x\in\rr$, then $[x]$ contains $x$ only. In this case, the
$(n-1)$-dimensional sphere has radius equal to zero.
The domains on which slice hyperholomorphic functions have a Cauchy formula are the so called slice domains and axially symmetric domains.
\begin{definition}[Slice domains and axially symmetric domains]
Slice domains: Let $U \subseteq \mathbb{R}^{n+1}$ be a domain. We
say that $U$ is a \textnormal{slice domain (s-domain for short)}
\index{s-domain} if $U\cap \mathbb{R}$ is nonempty and if
$\mathbb{C}_I\cap U$ is a domain in $\mathbb{C}_I$ for all $I \in \mathbb{S}$.
\\
Let $U \subseteq \mathbb{R}^{n+1}$. We say that $U$ is
\textnormal{axially symmetric} if, for all $u+Iv \in U$, the whole
$(n-1)$-sphere $[u+Iv]$ is contained in $U$.
\end{definition}

 It is important to point out that a key tool in our theory is the Cauchy formula for slice monogenic functions.  If
 $x=x_0+e_1x_1+...+e_nx_n$, $s=s_0+e_1s_1+...+e_ns_n$ are paravectors in $\rr^{n+1}$,
 then we have the following facts.
 \begin{proposition}
Suppose that $x$ and $s\in\rr^{n+1}$ are such that $x\not\in [s]$.
The following identity holds:
\begin{equation}\label{second}
-(x^2 -2x {\rm Re} (s)+|s|^2)^{-1}(x-\overline s)=(s-\bar x)(s^2-2{\rm
Re}(x)s+|x|^2)^{-1}.
\end{equation}
\end{proposition}
This fact justifies the following definition.
\begin{definition}
Let $x$, $s\in \rr^{n+1}$
such that $x\not\in [s]$.
\begin{itemize}
\item
We say that  $S_L^{-1}(s,x)$ is written in the form I if
$$
S_L^{-1}(s,x):=-(x^2 -2x {\rm Re} (s)+|s|^2)^{-1}(x-\overline s).
$$
\item
We say that $S_L^{-1}(s,x)$ is written in the form II if
$$
S_L^{-1}(s,x):=(s-\bar x)(s^2-2{\rm Re}(x) s+|x|^2)^{-1}.
$$
\end{itemize}
\end{definition}

\begin{proposition}
Suppose that $x$ and $s\in\rr^{n+1}$ are such that $x\not\in [s]$.
The following identity holds:
\begin{equation}\label{third}
 (s^2-2{\rm Re}(x)s+|x|^2)^{-1}(s-\bar x)=-(x-\bar s)(x^2-2{\rm Re}(s)x+|s|^2)^{-1} .
\end{equation}
\end{proposition}

This fact justifies the following definition.
\begin{definition}
Let $x$, $s\in \rr^{n+1}$
such that $x\not\in [s]$.
\begin{itemize}
\item
We say that  $S_R^{-1}(s,x)$ is written in the form I if
$$
S_R^{-1}(s,x):=-(x-\bar s)(x^2-2{\rm Re}(s)x+|s|^2)^{-1} .
$$
\item
We say that $S_R^{-1}(s,x)$ is written in the form II if
$$
S_R^{-1}(s,x):=(s^2-2{\rm Re}(x)s+|x|^2)^{-1}(s-\bar x).
$$
\end{itemize}
\end{definition}

\begin{theorem}[The Cauchy formula]
\label{Cauchygenerale}
Let $U\subset\mathbb{R}^{n+1}$ be an axially symmetric s-domain.
Suppose that $\partial (U\cap \mathbb{C}_I)$ is a finite union of
continuously differentiable Jordan curves  for every $I\in\mathbb{S}$.  Set  $ds_I=-ds I$ for $I\in \mathbb{S}$.
If $f$ is
a (left) slice monogenic function on a set that contains $\overline{U}$ then
\begin{equation}\label{cauchynuovo}
 f(x)=\frac{1}{2 \pi}\int_{\partial (U\cap \mathbb{C}_I)} S_L^{-1}(s,x)ds_I f(s)
\end{equation}
 and the
integral  depends neither on $U$ and nor on the imaginary unit
$I\in\mathbb{S}$.
\noindent
If $f$ is a right slice monogenic function on a set that contains $\overline{U}$,
then
\begin{equation}\label{Cauchyright}
 f(x)=\frac{1}{2 \pi}\int_{\partial (U\cap \mathbb{C}_I)}  f(s)ds_I S_R^{-1}(s,x)
 \end{equation}
and the integral  depends neither on $U$ and nor on the imaginary unit $I\in\mathbb{S}$.
\end{theorem}
The deepest property of of slice monogenic functions on axially symmetric slice domains is the  Representation Formula
(also called Structure Formula, see \cite{CSTrend}).
\begin{theorem}[Representation Formula]\label{formulamon}
\index{Representation Formula}
Let $U\subseteq
\mathbb{R}^{n+1}$ be an axially symmetric s-domain.
\begin{itemize}
\item
Let $f\in \mathcal{SM}^L(U)$.
 Then, for any  vector
$x =u+I_x v\in U$, we have
\begin{equation}\label{distributionmon}
f(x)=\frac{1}{2}\Big[1 -  I_x I\Big]f(u +I v)+\frac{1}{2}\Big[1 +  I_x I\Big]f(u -I v)
\end{equation}
and
\begin{equation}\label{distributionmon1}
f(x)=\frac{1}{2}\Big[   f(u +I v)+f(u -I v) +I_x
I[f(u -I v)-f(u +I v)]\Big].
\end{equation}
\item
Let $f\in \mathcal{SM}^R(U)$.
 Then, for any  vector
$x =u+I_x v\in U$, we have
\begin{equation}\label{distributionmonRG}
f(x)=\frac{1}{2}\Big[1 -  II_x\Big]f(u +I v)+\frac{1}{2}\Big[1 +  II_x\Big]f(u -I v)
\end{equation}
and
\begin{equation}\label{distributionmon1RG}
f(x) =\frac{1}{2}\Big[   f(u+Iv)+f(u-Iv)\Big]+\frac{1}{2}\Big[ [f(u-Iv)-f(u+Iv)]II_x\Big].
\end{equation}
\end{itemize}
\end{theorem}
As we will see in the sequel the Representation Formula shows that it is possible to apply the Fueter-Sce mapping theorem to slice monogenic functions to obtain monogenic functions, that is functions in the kernel of the Dirac operator.

\section{The Fueter-Sce mapping theorem in integral form}

For sake of completeness we recall the  notion of monogenic functions.
\begin{definition}[Monogenic functions] Let $U$ be an open set in $\mathbb{R}^{n+1}$.
A real differentiable function $f: U\to \mathbb{R}_n$ is left monogenic if
$$
\frac{\pp}{\pp x_0}f(x)+\sum_{i=1}^n e_i\frac{\pp}{\pp x_i}f(x)=0.
$$
It is right monogenic if
$$
\frac{\pp}{\pp x_0}f(x)+\sum_{i=1}^n \frac{\pp}{\pp x_i}f(x)e_i=0.
$$
\end{definition}
The Representation Formula shows that a slice monogenic functions   $f: U\subset \rr^{n+1}\to  \rr_n$ is of the form
$$
f(x+Iy)=\alpha(x,y)+I\beta (x,y),
$$
where $\alpha$ and $\beta$ are suitable  $\rr_n$--valued functions satisfying the Cauchy--Riemann system and $I$ is a $1$-vector in the Clifford algebra $\rr_n$ such that $I^2=-1$.
The Fueter-Sce theorem states that, for $n$ odd, if $f(x+Iy)=\alpha(x,y)+I\beta (x,y)$ is slice monogenic
then the function
$$
\breve{f}(x_0,|\underline{x}|)=
\Delta^{(n-1)/2}(\alpha(x_0,|\underline{x}|)+\frac{\underline{x}}{|\underline{x}|}\beta(x_0,|\underline{x}|))
$$
is monogenic, that is it is in the kernel of the Dirac operator. Here $\Delta$ is the Laplace operator in dimension $n+1$.
This means that using the Cauchy formula of the slice monogenic functions
we apply the operator $\Delta^{(n-1)/2}$ to the Cauchy kernels and we obtain
 an integral version of the Fueter-Sce mapping theorem.

The crucial point, as it has already been observed in \cite{F_FUNC_UNB}, is that we can get an elegant formula only when we apply the operator $\Delta^{(n-1)/2}$ to the Cauchy kernel in form II. This means that, when we define the $F$-functional calculus using this integral transform, we are restricted to the case of commuting operators.

\begin{theorem}\label{Laplacian_comp}
Let $x$,
$s\in \rr^{n+1}$ be such that $x\not\in [s]$ and let
 $\Delta=\sum_{i=0}^n\frac{\partial^2}{\partial x_i^2}$ be the Laplace operator in the variable $x$.
\begin{itemize}
\item[(a)]
Consider the left slice monogenic Cauchy kernel $S_L^{-1}(s,x)$ written in the form II, that is
$$
S_L^{-1}(s,x):=(s-\bar x)(s^2-2{\rm Re}(x) s+|x|^2)^{-1}.
$$
Then, for $h\geq 1$, we have:
\begin{equation}\label{hLaplacian}
\Delta^hS_L^{-1}(s,x)=(-1)^h\prod_{\ell=1}^h(2\ell) \prod_{\ell=1}^h (n-(2\ell -1))(s-\bar x)(s^2-2{\rm Re}[x]s +|x|^2)^{-(h+1)}.
\end{equation}
\item[(b)]
Consider the right slice monogenic Cauchy kernel $S_R^{-1}(s,x)$ written in the form II, that is
$$
S_R^{-1}(s,x):=(s^2-2{\rm Re}(x)s+|x|^2)^{-1}(s-\bar x).
$$
Then, for $h\geq 1$, we have:
\begin{equation}\label{hLaplacianR}
\Delta^hS_R^{-1}(s,x)=(-1)^h\prod_{\ell=1}^h(2\ell) \prod_{\ell=1}^h(n-(2\ell -1)) (s^2-2{\rm Re}[x]s +|x|^2)^{-(h+1)}(s-\bar x).
\end{equation}
\end{itemize}
\end{theorem}

\begin{proof}
We will only consider the right slice monogenic case as the left slice monogenic case has already been proved in \cite{CoSaSo}. We will proof the formula by induction. We have
\begin{multline*}
\frac{\partial^2}{\partial x_0^2} S_R^{-1}(s,x) = 2(s^2-2{\rm Re}(x) s + |x|^2)^{-3} (-2s+2x_0)^2(s-\bar x)\\
-2(s^2 - 2{\rm Re}(x) s + |x|^2)^{-2}(s-\bar{x}) + 2(s^2 - 2{\rm Re}(x)s + |x|^2)^{-2}(-2s + 2x_0)
\end{multline*}
and for $i=1,\ldots,n$
\begin{multline*}
\frac{\partial^2}{\partial x_i^2}S_R^{-1}(s,x) = 8x_i^2(s^2-2{\rm Re}(x)s + |x|^2)^{-3}(s-\bar{x})  \\
-2(s^2 - 2{\rm Re}(x)s + |x|^2)^{-2}(s-\bar{x}) - 4x_i(s^2 - 2{\rm Re}(x)s + |x|^2)^{-2}e_i.
\end{multline*}
Thus we obtain
\begin{align*}
\Delta S_R^{-1}(s,x) &= 2(s^2-2{\rm Re}(x) s + |x|^2)^{-3} (-2s+2x_0)^2(s-\bar x)  \\
&+2(s^2 - 2{\rm Re}(x)s + |x|^2)^{-2}(-2s + 2x_0)
\\
&
+\sum_{i=1}^{n}8x_i^2(s^2-2{\rm Re}(x)s + |x|^2)^{-3}(s-\bar{x}) \\
 &- \sum_{i=1}^{n}4x_i(s^2 - 2{\rm Re}(x)s + |x|^2)^{-2}e_i -2(n+1)(s^2 - 2{\rm Re}(x)s + |x|^2)^{-2}(s-\bar{x})
\end{align*}
and since $(s^2 - 2{\rm Re}(x)s + |x|^2)^{-1}$ and $(-2s + 2x_0)$ commute we have
\begin{align*}
\Delta S_R(s,x)^{-1} =& \left(2(-2s+2x_0)^2+\sum_{i=1}^{n}8x_i^2\right)(s^2-2{\rm Re}(x) s + |x|^2)^{-3} (s-\bar x)  \\
&+(s^2 - 2{\rm Re}(x)s + |x|^2)^{-2}\left(2(-2s + 2x_0) - \sum_{i=1}^{n}4x_ie_i\right)\\
 & -2(n+1)(s^2 - 2{\rm Re}(x)s + |x|^2)^{-2}(s-\bar{x}). \\
 \end{align*}
 Finally we obtain
 \begin{align*}
\Delta S_R^{-1}(s,x)=& 8(s^2 -2 {\rm Re}(x)s + |x|^2)^{-2} (s-\bar{x})\\
&- 4(s^2 - 2{\rm Re}(x)s + |x|^2)^{-2}(s-\bar{x}) \\
&-2(n+1)(s^2 - 2{\rm Re}(x)s + |x|^2)^{-2}(s-\bar{x})\\
=&-2(n-1)(s^2 - 2{\rm Re}(x)s + |x|^2)^{-2}(s-\bar{x})
\end{align*}
which corresponds to \eqref{hLaplacianR} for $h=1$.

Let us assume that the formula \eqref{hLaplacianR} holds for some $h\in\mathbb{N}$ and let us show that it holds for $h+1$. In order to avoid the constants, we consider the function
\begin{equation}
\label{sigmah}
G_h(s,x) := (s^2-2{\rm Re}(x)s + |x|^2)^{-(h+1)}(s-\bar{x}).
\end{equation}
We have
\begin{align*}
\frac{\partial^2}{\partial x_0^2}G_h =&  (h+2)(h+1)(s^2-2{\rm Re}(x)s + |x|^2)^{-(h+3)}(-2s+2x_0)^2(s-\bar{x})\\
&-2(h
+1)(s^2-2{\rm Re}(x)s + |x|^2)^{-(h+2)}(s-\bar{x})\\
&+ 2(h+1)(s^2-2{\rm Re}(x)s + |x|^2)^{-(h+2)}(-2s+2x_0)
\end{align*}
and
\begin{align*}
\frac{\partial^2}{\partial x_i^2}G_h(s,x) = &4(h+2)(h+1)(s^2-2{\rm Re}(x)s + |x|^2)^{-(h+3)}x_i^2(s-\bar{x}) \\
&-2(h+1)(s^2-2{\rm Re}(x)s + |x|^2)^{-(h+2)}(s-\bar{x}) \\
&-4(h+1)(s^2-2{\rm Re}(x)s + |x|^2)^{-(h+2)}x_ie_i.
\end{align*}
Thus, we obtain
\begin{equation*}
\Delta G_h(s,x) = -2(h+1)(n-(2h+1))(s^2-2{\rm Re}(x)s + |x|^2)^{-(h+2)}(s-\bar{x})
\end{equation*}
and so, taking into account the fact that
\begin{equation*}
\Delta S_R^{-1}(s,x) = (-1)^h\prod_{l=1}^{h}(2l)\prod_{l=1}^{h}(n-(2l-1))G_h(s,x),
\end{equation*}
we obtain that \eqref{hLaplacianR} holds for $h+1$. By induction we get the statement.

\end{proof}

\begin{proposition}\label{Laplacian_smonogenic}
Let $x$, $s\in \rr^{n+1}$ be such that $x\not\in [s]$.
The function $\Delta^hS_L^{-1}(s,x)$
is a right slice monogenic function in the variable $s$ for all $h\geq 0$.
The function $\Delta^hS_R^{-1}(s,x)$
is a left slice monogenic function in the variable $s$ for all $h\geq 0$.
\end{proposition}
\begin{proof}
We will only consider the right slice monogenic case as the left slice monogenic case has already been proved in \cite{CoSaSo}.
For $h=0$ the statement is well known. If $h\geq 1$, we set $s = u+Iv$ for $I\in\mathbb{S}$ and we consider the function $G_h$ introduced in \eqref{sigmah} to avoid the constants. We have
\begin{multline*}
\frac{\partial}{\partial u} G_h(u+vI,x) = (u^2 -v^2 + 2Iuv +2{\rm Re}(x)(u+Iv) +|x|^2)^{-(h+1)}
\\ -(h+1)(u^2 -v^2 + 2Iuv +2{\rm Re}(x)(u+Iv) +|x|^2)^{-(h+2)}(2u+2Iv-2{\rm Re}(x_0))(u+vI-\bar{x})
\end{multline*}
and
\begin{multline*}
\frac{\partial}{\partial u} G_h(u+vI,x) = (u^2 -v^2 + 2Iuv +2{\rm Re}(x)(u+Iv) +|x|^2)^{-(h+1)}
\\ -(h+1)(u^2 -v^2 + 2Iuv +2{\rm Re}(x)(u+Iv) +|x|^2)^{-(h+2)}(-2v+2Iu-2I{\rm Re}(x_0))(u+vI-\bar{x}).
\end{multline*}
As $I$ and $(u^2 -v^2 + 2Iuv +2{\rm Re}(x)(u+Iv) +|x|^2)^{-1}$ commute, it follows immediately that
\begin{equation*}
\frac{\partial}{\partial u} G_h(u+Iv, x) + I\frac{\partial}{\partial v}G_h(u+Iv,x) = 0.
\end{equation*}
Therefore $\Delta S_R^{-1}(s,x)$ is left slice monogenic in its first variable.

\end{proof}

\begin{proposition}\label{Laplacian}
Let $n$ be an odd number and
let $x$,
$s\in \rr^{n+1}$
be such that
 $x\not\in [s]$.
Then the function $\Delta^{\frac{n-1}{2}}S_L^{-1}(s,x)$ is a left monogenic function in the variable $x$,
and the function  $\Delta^{\frac{n-1}{2}}S_R^{-1}(s,x)$ is a right monogenic function in the variable $x$.
\end{proposition}

\begin{proof}
We will only consider the right slice monogenic case as the left slice monogenic case has already been proved in \cite{CoSaSo}.
Again we consider the function $G_{\frac{n-1}{2}}$ as defined in \eqref{sigmah} to avoid the constants. We have
\begin{align*}
\frac{\partial}{\partial x_0} G_{\frac{n-1}{2}} = &-\frac{n+1}{2}(s^2 - 2{\rm Re}(x)s+|x|^2)^{-\frac{n+1}{2}-1}(-2s+2x_0)(s-\bar{x})\\
&- (s^2 - 2{\rm Re}(x)s+|x|^2)^{-\frac{n+1}{2}}
\end{align*}
and
\begin{align*}
\frac{\partial}{\partial x_i}G_{\frac{n-1}{2}} = &-\frac{n+1}{2}(s^2 - 2{\rm Re}(x)s+|x|^2)^{-\frac{n+1}{2}-1}(2x_i)(s-\bar{x})\\
&+(s^2 - 2{\rm Re}(x)s+|x|^2)^{-\frac{n+1}{2}}e_i.
\end{align*}
for $i = 1,\ldots,n$.
If we consider $\partial_rG_{\frac{n-1}{2}} := \frac{\partial}{\partial x_{0}}G_{\frac{n-1}{2}} + \sum_{i=1}^{n}(\frac{\partial}{\partial x_i}G_{\frac{n-1}{2}})e_i $ we get
\begin{align*}
\partial_{r}G_{\frac{n-1}{2}}  = &-\frac{n+1}{2}(s^2 - 2{\rm Re}(x)s+|x|^2)^{-\frac{n+1}{2}-1}(-2s+2x_0)(s-\bar{x})\\
&- (s^2 - 2{\rm Re}(x)s+|x|^2)^{-\frac{n+1}{2}} \\
&-\frac{n+1}{2}(s^2 - 2{\rm Re}(x)s+|x|^2)^{-\frac{n+1}{2}-1}(s-\bar{x})\left(\sum_{i=1}^{n}2x_ie_i\right)  \\
&-n (s^2 - 2{\rm Re}(x)s+|x|^2)^{-\frac{n+1}{2}} \\
=& (n+1)(s^2 - 2{\rm Re}(x)s + |x|^2)^{-\frac{n+1}{2} -1}\left(s(s-\bar{x}) - (s-\bar{x})x\right) \\
&-(n+1)(s^2 - 2{\rm Re}(x)s + |x|^2)^{-\frac{n+1}{2}} = 0
\end{align*}
as $s(s-\bar{x})-(s-\bar{x})x = s^2 - 2{\rm Re}(x)s + |x|^2$. Therefore $G_{\frac{n-1}{2}}(s,x)$ and $S_R^{-1}(s,x)$ are right monogenic in $x$.

\end{proof}

\begin{definition}[The $F_n$-kernel]
Let $n$ be an odd number and  let $x$, $s\in \rr^{n+1}$.
We define, for $s\not\in[x]$, the $F_n^L$-kernel as
$$
F_n^L(s,x):=\Delta^{\frac{n-1}{2}}S_L^{-1}(s,x)=\gamma_n(s-\bar x)(s^2-2{\rm Re}(x)s +|x|^2)^{-\frac{n+1}{2}},
$$
and the $F_n^R$-kernel as
$$
F_n^R(s,x):=\Delta^{\frac{n-1}{2}}S_R^{-1}(s,x)=\gamma_n(s^2-2{\rm Re}(x)s +|x|^2)^{-\frac{n+1}{2}}(s-\bar x),
$$
where
\begin{equation}\label{gamman}
\gamma_n:=(-1)^{(n-1)/2}2^{n-1}\Big[\Big( \frac{n-1}{2}\Big)!\Big]^2 .
\end{equation}
\end{definition}
\begin{remark}
{\rm
Observe that the constants $\gamma_n$ are obtained from the identity
$$
(-1)^{(n-1)/2}\prod_{\ell=1}^{(n-1)/2}(2\ell) \prod_{\ell=1}^{(n-1)/2} (n-(2\ell -1))
=(-1)^{(n-1)/2}2^{n-1}\Big[\Big( \frac{n-1}{2}\Big)!\Big]^2.
$$
}
\end{remark}

\begin{theorem}[The Fueter-Sce mapping theorem in integral form]
Let $n$ be an odd number and set $ds_I=ds/ I$.
Let $W\subset \mathbb{R}^{n+1}$ be an open set. Let $U$ be a bounded  axially symmetric s-domain such that $\overline{U}\subset W$.
 Suppose that the boundary of $U\cap \mathbb{C}_I$  consists of a
finite number of rectifiable Jordan curves for any $I\in\mathbb{S}$.
\begin{itemize}
\item[(a)]
If $x\in U$ and $f\in \mathcal{SM}^L(W)$ then
$\breve{f}(x)=\Delta^{\frac{n-1}{2}}f(x)$
is left monogenic and it admits the integral representation
\begin{equation}\label{Fueter}
 \breve{f}(x)=\frac{1}{2 \pi}\int_{\pp (U\cap \mathbb{C}_I)} F_n^L(s,x)ds_I f(s).
\end{equation}
\item[(b)]
If $x\in U$ and $f\in \mathcal{SM}^R(W)$ then
$\breve{f}(x)=\Delta^{\frac{n-1}{2}}f(x)$
is right monogenic and it admits the integral representation
\begin{equation}\label{Fueter}
 \breve{f}(x)=\frac{1}{2 \pi}\int_{\pp (U\cap \mathbb{C}_I)} f(s)ds_I F_n^R(s,x).
\end{equation}
\end{itemize}
The integrals  depend neither on $U$ and nor on the imaginary unit $I\in\mathbb{S}$.
\end{theorem}
\begin{proof}
It follows from the Fueter- Sce mapping theorem, from the Cauchy formulas and from  Proposition \ref{Laplacian}.
\end{proof}

We point out that the Fueter-Sce mapping theorem in integral form can be proved for more general open sets, precisely
 for open sets that are only axially symmetric, see \cite{CoSaSo}.
 For the application to the $F$-functional calculus  we can take axially symmetric s-domains,
 this is the reason for which we work directly with slice monogenic functions.
 In this case we consider just axially symmetric open sets we have to change the definition of slice monogenicity and  consider holomorphic functions of a paravector variable.
Precisely the definition has to changed as follows.
 Let $U\subseteq\rr^{n+1}$ be an axially symmetric open set
 and let $\mathcal{U}\subseteq\mathbb{R}\times\rr$ be such that $x=u+Iv\in U$ for all $(u,v)\in\mathcal{U}$.
We consider  functions on $U$ of the form
$$f(x)=\alpha(u,v)+I\beta(u,v)$$
 where $\alpha$, $\beta$ are $\mathbb{R}_n$-valued differentiable functions such that
 $$
 \alpha(u,v)=\alpha(u,-v), \ \ \ \beta(u,v)=\beta(u,-v)\ \ \ {\rm for \ all}\ \  (u,v)\in \mathcal{U}
 $$
 and moreover $\alpha$ and $\beta$ satisfy the Cauchy-Riemann system
 $$
\pp_u \alpha-\pp_v\beta=0,\ \ \ \ \
\pp_v\alpha+\pp_u \beta=0.
$$
On axially symmetric s-domains this class of functions coincides with the class of slice monogenic functions.

\section{The formulations of the $F$-functional calculus}

In this section we recall some definitions and results to be used in the sequel.
By $V$ we denote a real Banach space over $\mathbb{R}$ with norm $\|\cdot \|$.
It is possible to endow $V$ with an operation of multiplication by elements of $\mathbb{R}_n$ which gives
a two-sided module over $\mathbb{R}_n$.
  We denote by $V_n$  two-sided Banach module $V\otimes \mathbb{R}_n$.
 An element in $V_n$ is of the form $\sum_A v_A\otimes e_A$ (where
 $A=i_1\ldots i_r$, $i_\ell\in \{1,2,\ldots, n\}$, $i_1<\ldots <i_r$ is a multi-index).
The multiplications (right and left) of an element $v\in V_n$ with a scalar
$a\in \mathbb{R}_n$ are defined as
$$
va=\sum_A v_A \otimes (e_A a),\ \ \ \ {\rm and}\ \ \ \ av=\sum_A v_A \otimes (ae_A ).
$$
For short, we will write
$\sum_A v_A e_A$ instead of $\sum_A v_A \otimes e_A$. Moreover, we define
$$\| v\|_{V_n}=
\sum_A\| v_A\|_V.$$

Let $\mathcal{B}(V)$ be the space of bounded $\mathbb{R}$-homomorphisms of the Banach space $V$ into itself
 endowed with the natural norm denoted by $\|\cdot\|_{\mathcal{B}(V)}$.
\noindent
If $T_A\in \mathcal{B}(V)$, we can define the operator $T=\sum_A T_Ae_A$ and
its action on
$$
v=\sum_B v_Be_B
$$
 as
 $$
 T(v)=\sum_{A,B}T_A(v_B)e_Ae_B.
 $$
\noindent
The set of all such bounded operators is denoted by $\mathcal{B}(V_n)$. The norm is defined by
$$
\|T\|_{\mathcal{B}(V_n)}=\sum_A \|T_A\|_{\mathcal{B}(V)}.
$$
In the sequel,we will only consider operators of the form
$T=T_0+\sum_{j=1}^ne_jT_j$ where $T_j\in\mathcal{B}(V)$ for $j=0,1,...,n$ and we recall that
 the conjugate  is defined by $\overline{T}=T_0-\sum_{j=1}^ne_jT_j$.
The set of such operators in ${\mathcal{B}(V_n)}$ will be denoted by $\mathcal{B}^{\small 0,1}(V_n)$.
In the section we will always consider n-tuples of bounded commuting operators, in paravector form,
and we will denote such set as $\mathcal{BC}^{\small 0,1}(V_n)$.
We recall some results proved is \cite{SlaisComm} and in \cite{CoSaSo}.
\begin{definition}[The $F$-spectrum and the $F$-resolvent sets]\label{seesspect}
Let
$T\in\mathcal{BC}^{\small 0,1}(V_n)$.
We define the $F$-spectrum $\sigma_F(T)$ of $T$ as:
$$
\sigma_{F}(T)=\{ s\in \mathbb{R}^{n+1}\ \ :\ \ \ \ \ s^2\mathcal{I}-s(T+\overline{T})+T\overline{T}\ \ \
{\it is\ not\  invertible}\}.
$$
The $F$-resolvent set $\rho_F(T)$ is defined by
$$
\rho_{F}(T)=\mathbb{R}^{n+1}\setminus\sigma_{F}(T).
$$
\end{definition}
Here we state two important properties of the $F$-spectrum. The fact that it has axial symmetry and for the case of bounded operators it is a compact and non empty set.
\begin{theorem}[Structure of the $F$-spectrum]\label{strutturaK}
 Let $T\in\mathcal{BC}^{\small 0,1}(V_n)$ and let $p = p_0 +p_1I\in
  [p_0 +p_1 I]\subset \mathbb{R}^{n+1}\setminus\mathbb{R}$, such that $p\in \sigma_{ F}(T)$.
  Then all the elements of the $(n-1)$-sphere $[p_0 +p_1I]$
 belong to $\sigma_{ F}(T)$.
\end{theorem}
Thus the $F$-spectrum consists of real points and/or $(n-1)$-spheres.
\begin{theorem}[Compactness of the $F$-spectrum]\label{compattezaDS}
\par\noindent
Let
$T\in\mathcal{BC}^{\small 0,1}(V_n)$. Then
the $F$-spectrum $\sigma_{F} (T)$  is a compact nonempty set.
Moreover, $\sigma_{F} (T)$ is contained in $\{
s\in\mathbb{R}^{n+1}\, :\, | s|  \leq \|T\| \ \}$.
\end{theorem}
The definition of $F$-resolvent operator is suggested by the Fueter-Sce mapping theorem in integral form.
\begin{definition}[$F$-resolvent operators]
Let $n$ be an odd number and let
$T\in\mathcal{BC}^{\small 0,1}(V_n)$.
 For $s\in \rho_F(T)$
we define the left $F$-resolvent operator by
\begin{equation}\label{FresBOUNDL}
F_n^L(s,T):=\gamma_n(s\mathcal{I}-\overline{ T})(s^2\mathcal{I}-s(T+\overline{T})+T\overline{T})^{-\frac{n+1}{2}},
\end{equation}
and the right $F$-resolvent operator by
\begin{equation}\label{FresBOUNDR}
F_n^R(s,T):=\gamma_n(s^2\mathcal{I}-s(T+\overline{T})+T\overline{T})^{-\frac{n+1}{2}}(s\mathcal{I}-\overline{ T}),
\end{equation}
where the constants $\gamma_n$ are given by (\ref{gamman}).
\end{definition}
Now we have to say which are the functions that are defined on suitable open sets that contain the $F$-spectrum. For this class of slice
monogenic functions it is possible to define the $F$-functional calculus for bounded operators.
\begin{definition}\label{def3.9_mondsa}
Let  $T\in\mathcal{BC}^{\small 0,1}(V_n)$  and let $U \subset \mathbb{R}^{n+1}$ be an axially symmetric s-domain.
\begin{itemize}
\item[(a)] We say that $U$ is admissible for $T$ if it contains  the $F$-spectrum $\sigma_F(T)$, and if
$\pp (U\cap \mathbb{C}_I)$ is the union of a finite number of
rectifiable Jordan curves  for every $I\in\mathbb{S}$.
\item[(b)]
Let $W$ be an open set in $\mathbb{R}^{n+1}$.
A function  $f\in \mathcal{SM}^L(W)$ (resp. $f\in \mathcal{SM}^R(W)$) is said to be locally  left (resp. right) slice monogenic  on $\sigma_{F}(T)$
if there exists an admissible domain $U\subset \mathbb{R}^{n+1}$  such that $\overline{U}\subset W$, on
which $f$ is left (resp. right) slice monogenic.
\item[(c)]
We will denote by $\mathcal{SM}^L_{\sigma_{F}(T)}$ (resp. right $\mathcal{SM}^R_{\sigma_{F}(T)}$)  the set of locally left (resp. right) slice monogenic  functions on $\sigma_{F}(T)$.
\end{itemize}
\end{definition}
Finally the following theorem is a crucial fact for the well posedness of the $F$-functional calculus.
\begin{theorem} Let $n$ be an odd number, let
 $T\in\mathcal{BC}^{\small 0,1}(V_n)$ and set $ds_I=ds/I$ .
   Then the integrals
\begin{equation}\label{integ311_mon}
\frac{1}{2\pi}\int_{\pp(U\cap \mathbb{C}_I)} F_n^L(s,T) \, ds_I\, f(s),\ \ \ \ \ f\in \mathcal{SM}^L_{\sigma_{F}(T)},
\end{equation}
and
\begin{equation}\label{integ311_monRIGHT}
\frac{1}{2\pi}\int_{\pp(U\cap \mathbb{C}_I)} f(s) \, ds_I\, F_n^R(s,T),\ \ \ \ \ f\in \mathcal{SM}^R_{\sigma_{F}(T)},
\end{equation}
 depend neither on the imaginary unit $I\in\mathbb{S}$ and nor on the set $U$.
\end{theorem}
\begin{proof}
The case of left slice monogenic functions has been treated in \cite{CoSaSo}.
We only consider the right slice monogenic case. For every continuous linear functional $\phi\in V_n'$ and $v\in V_n$, we define the function
\begin{equation*}
g_{\phi,v}(s) := \left<\phi, F_n^R(s,T)v\right>.
\end{equation*}
Since $\Delta^{\frac{n-1}{2}}S^{-1}_R$ is left s-monogenic in $s$ by Proposition \ref{Laplacian_smonogenic}, the function $g_{\phi,v}$ is left slice monogenic too on $\rho_F(T)$. Furthermore, as $\lim_{s\to\infty}g_{\phi,v}(s) = 0$,  it is also left slice monogenic at $\infty$.

To proof that the integral \eqref{integ311_monRIGHT} does not depend on the open set $U$, we consider another open set $U'$ as in Definition \ref{def3.9_mondsa} with $\sigma_F(T) \subset U' \subset U$. As  every $g_{\phi,v}$ is left slice monogenic on $U'^c\subset\rho_F(T)$ and $g_{\phi,v}(\infty) = 0$, we can apply the Cauchy formula and obtain
\begin{equation}
\label{prf:fc_g_int}
g_{\phi,v}(x) = \frac{1}{2\pi}\int_{\partial(U'\cap\mathbb{C}_I)^-}S_L^{-1}(s,x)ds_I g_{\phi,v}(s) = - \frac{1}{2\pi}\int_{\partial(U'\cap\mathbb{C}_I)}S_L^{-1}(s,x)ds_I g_{\phi,v}(s),
\end{equation}
where $\partial(U'\cap\mathbb{C})^{-}$ is the border of $U'\cap\mathbb{C}$ oriented in the way that includes $U'^{c}$. As $S_R^{-1}(s,x) = - S_L^{-1}(x,s)$ and as $f$ is right slice monogenic on $\overline{U}$, we get
\[
\begin{split}
\big<\phi, \Big[\frac{1}{2\pi}&\int_{\partial(U\cap\mathbb{C}_I)} f(s) ds_I F_n^{R}(s,T)\Big] v\big>
\\
&
= \frac{1}{2\pi}\int_{\partial(U\cap\mathbb{C}_I)}f(s)ds_Ig_{\phi,v}(s)
\\
&
=\frac{1}{2\pi}\int_{\partial(U\cap\mathbb{C}_I)}f(s)ds_I\left[ -\frac{1}{2\pi}\int_{\partial(U'\cap\mathbb{C}_I)^-}S_L^{-1}(t,s)dt_I g_{\phi,v}(t) \right]
\\
&
= \frac{1}{2\pi}\int_{\partial(U'\cap\mathbb{C}_I)}\left[\frac{1}{2\pi}
\int_{\partial(U\cap\mathbb{C}_I)}f(s)ds_IS_R^{-1}(s,t)\right]dt_Ig_{\phi,v}(t)
\\
&
= \frac{1}{2\pi}\int_{\partial(U'\cap \mathbb{C}_I)}f(t)dt_Ig_{\phi,v}(t)
\\
&
= \big< \phi, \left[\frac{1}{2\pi}\int_{\partial(U'\cap\mathbb{C}_I)}f(s)ds_IF_n^{R}(s,T)\right]v\big>.
\end{split}
\]
This equality holds for every $\phi\in V_n'$ and $v\in V_n$. Therefore, by the Hahn-Banach theorem if follows that
\begin{equation*}
\frac{1}{2\pi}\int_{\partial(U\cap\mathbb{C}_I)}f(s)ds_IF_n^{R}(s,T)= \frac{1}{2\pi}\int_{\partial(U'\cap\mathbb{C}_I)}f(s)ds_IF_n^{R}(s,T).
\end{equation*}
Now let $\tilde{U}$ be an arbitrary set as in Definition \ref{def3.9_mondsa} that is not necessarily a subset of U. Then we can find an admissible set $U'$ with $\sigma_F(T) \subset U' \subset U\cap\tilde{U}$ and therefore the integrals over all three sets must agree.

The proof of the independence of the imaginary unit $I$ works analogously. Again we consider an admissible open set $U'$ with $\sigma_F(T)\subset U' \subset U$.  As \eqref{prf:fc_g_int} is independent of the imaginary unit, for an arbitrary $J\in\mathbb{S}$ we have
\[
\begin{split}
\frac{1}{2\pi}\int_{\partial(U\cap\mathbb{C}_I)}f(s)ds_Ig_{\phi,v}(s)
&=\frac{1}{2\pi}\int_{\partial(U\cap\mathbb{C}_I)}f(s)ds_I\left[ -\frac{1}{2\pi}\int_{\partial(U'\cap\mathbb{C}_J)^-}S_L^{-1}(t,s)dt_J g_{\phi,v}(t) \right]
\\
&
= \frac{1}{2\pi}\int_{\partial(U'\cap\mathbb{C}_J)}
\left[\frac{1}{2\pi}\int_{\partial(U\cap\mathbb{C}_I)}f(s)ds_IS_R^{-1}(s,t)\right]dt_Jg_{\phi,v}(t)
\\
&
= \frac{1}{2\pi}\int_{\partial(U'\cap \mathbb{C}_J)}f(t)dt_Jg_{\phi,v}(t).
\end{split}
\]
As we already know that these integrals are independent of the set $U$. Therefore, for every $\phi\in V_n'$ and $v\in V_n$, we have
\begin{equation*}
\frac{1}{2\pi}\int_{\partial(U\cap\mathbb{C}_I)}f(s)ds_Ig_{\phi,v}(s)= \frac{1}{2\pi}\int_{\partial(U\cap \mathbb{C}_J)}f(t)dt_Jg_{\phi,v}(t).
\end{equation*}
and from the Hahn Banach theorem if follows that  the integral \eqref{integ311_monRIGHT} does not depend on the imaginary unit.

\end{proof}

\begin{definition}[The $F$-functional calculus for bounded operators]  Let $n$ be an odd number, let
 $T\in\mathcal{BC}^{\small 0,1}(V_n)$ and set $ds_I=ds/I$ .
   We define the  $F$-functional calculus as
\begin{equation}\label{DEFinteg311_mon}
\breve{f}(T):=\frac{1}{2\pi}\int_{\pp(U\cap \mathbb{C}_I)} F_n^L(s,T) \, ds_I\, f(s),\ \ \ \ \ f\in \mathcal{SM}^L_{\sigma_{F}(T)},
\end{equation}
and
\begin{equation}\label{DEFinteg311_monRIGHT}
\breve{f}(T):=\frac{1}{2\pi}\int_{\pp(U\cap \mathbb{C}_I)} f(s) \, ds_I\, F_n^R(s,T),\ \ \ \ \ f\in \mathcal{SM}^R_{\sigma_{F}(T)},
\end{equation}
where $U$ is admissible for $T$.
\end{definition}

\section{Some relations between the $F$-resolvent operators}
With the position
$$
\mathcal{Q}_s(T):= (s^2 \mathcal{I}-s(T+\overline{T})+T\overline{T})^{-1}, \ \ \ \ s\in \rho_F(T)
$$
we can write the $F$-resolvent operators as
\begin{equation}\label{FresBOUNDLQ}
F_n^L(s,T):=\gamma_n(s\mathcal{I}-\overline{ T})
\mathcal{Q}_s(T)^{\frac{n+1}{2}}
\end{equation}
and the right $F$-resolvent operator as
\begin{equation}\label{FresBOUNDRQ}
F_n^R(s,T):=\gamma_n
\mathcal{Q}_s(T)^{\frac{n+1}{2}}(s\mathcal{I}-\overline{ T}).
\end{equation}

\begin{theorem}[The left and the right $F$-resolvent equations]
\label{FRE}
Let $n$ be an odd number and let $T\in\mathcal{BC}^{0,1}(V_n)$.
Let $s\in \rho_F(T)$ then the $F$-resolvent operators  satisfy the equations
\begin{equation}\label{FREL}
F_n^L(s,T)s-TF_n^L(s,T)=\gamma_n\mathcal{Q}_s(T)^{\frac{n-1}{2}},
\end{equation}
and
\begin{equation}\label{FRER}
sF_n^R(s,T)-F_n^R(s,T)T=\gamma_n\mathcal{Q}_s(T)^{\frac{n-1}{2}}.
\end{equation}
\end{theorem}
\begin{proof}
Relation (\ref{FREL}) was proved in \cite{NOTIONCONV}. The second relation follows from
$$
F_n(s,T)s=\gamma_n(s\mathcal{I}-\overline{ T})s
\mathcal{Q}_s(T)^{\frac{n+1}{2}}
$$
and
$$
TF_n(s,T)=\gamma_n( Ts-T\overline{ T})
\mathcal{Q}_s(T)^{\frac{n+1}{2}}.
$$
Taking the difference we obtain
$$
F_n(s,T)s-TF_n(s,T)=\gamma_n(s^2 \mathcal{I}-s(T+\overline{T})+T\overline{T})\mathcal{Q}_s(T)^{\frac{n+1}{2}} = \gamma_n\mathcal{Q}_s(T)^{\frac{n+1}{2}}.
$$
\end{proof}

\begin{theorem}[Left and right generalized $F$-resolvent equations]
Let $n$ be an odd number, $T\in \mathcal{BC}^{0,1}(V_n)$ and set
\begin{multline*}
\mathcal{M}^L_m(s,T) := \gamma_n\sum_{i=0}^{m-1}T^i\mathcal{Q}_s(T)^{\frac{n-1}{2}}s^{m-1-i}=\\
=\gamma_n(\mathcal{Q}_s(T)^{\frac{n-1}{2}}s^{m-1} + T\mathcal{Q}_s(T)^{\frac{n-1}{2}}s^{m-2} + \ldots + T^{m-2}{Q}_s(T)^{\frac{n-1}{2}}s + T^{m-1}{Q}_s(T)^{\frac{n-1}{2}})
\end{multline*}
and
\begin{multline*}
\mathcal{M}^R_m(s,T) := \gamma_n\sum_{i=0}^{m-1}s^{m-1-i}\mathcal{Q}_s(T)^{\frac{n-1}{2}}T^{i}=\\
=\gamma_n(s^{m-1}\mathcal{Q}_s(T)^{\frac{n-1}{2}} + s^{m-2} \mathcal{Q}_s(T)^{\frac{n-1}{2}}T+ \ldots + s{Q}_s(T)^{\frac{n-1}{2}}T^{m-2} + T^{m-1}{Q}_s(T)^{\frac{n-1}{2}}).
\end{multline*}
Then for $m\in\mathbb{N}_0$ and $s\in\rho_F(T)$ the following equations hold
\begin{equation}
\label{GFREL}
T^mF_n^L(s,T)  = F_n^L(s,T)s^m  -\mathcal{M}^L_m(s,T),
\end{equation}
\begin{equation}
\label{GFRER}
F_n^R(s,T)T^m = s^m F_n^R(s,T) - \mathcal{M}^R_m(s,T).
\end{equation}
\end{theorem}
\begin{proof}
The proof works by induction. We will only show (\ref{GFREL}) as it works analogously for the second equation. For $m=0$ the statement is trivial, for $m=1$ it is the $F$-resolvent equation, Theorem \ref{FRE}.  Assume that \eqref{GFREL} holds for $m-1$. Then we have
\begin{equation}\label{GFRE-step1}
\begin{split}
T^{m}F_n^L(s,T)& = T(F_n^L(s,T)s^{m-1} - \mathcal{M}_{m-1}(s,T))
\\
&
= TF_n^L(s,T)s^{m-1} - T\mathcal{M}_{m-1}(s,T).
\end{split}
\end{equation}
As
\begin{equation*}
\begin{split}
T\mathcal{M}_{m-1}(s,T) &= \gamma_nT\sum_{i=0}^{m-2}T^{i}\mathcal{Q}_s(T)^{\frac{n-1}{2}}s^{m-2-i}
\\
&
= \gamma_n\sum_{i=1}^{m-1}T^{i}\mathcal{Q}_s(T)^{\frac{n-1}{2}}s^{m-1-i},
\end{split}
\end{equation*}
by applying the resolvent equation \eqref{FREL} in \eqref{GFRE-step1}, we obtain
\[
\begin{split}
T^mF_N^L(s,T) &= F_n^L(s,T)s^{m} - \gamma_nQ_s(T)^{\frac{n-1}{2}}s^{m-1} - \gamma_n\sum_{i=1}^{m-1}T^{i}\mathcal{Q}_s(T)^{\frac{n-1}{2}}s^{m-1-i}
\\
&
= F_n^L(s,T)s^{m} - \gamma_n\sum_{i=0}^{m-1}T^{i}\mathcal{Q}_s(T)^{\frac{n-1}{2}}s^{m-1-i}
\\
&
= F_n^L(s,T)s^m  -\mathcal{M}_m(s,T).
\end{split}
\]

\end{proof}

\begin{theorem}[The pseudo $F$-resolvent equation]\label{THPSEUDO}
Let $n$ be an odd number and let $T\in\mathcal{BC}^{0,1}(V_n)$.
Then for $p,s\in \rho_F(T)$ the following equation holds
\begin{multline*}
F_n^{R}(s,T)F_n^{L}(p,T) = \left[ \left( F^R_n(s,T) \gamma_n \mathcal{Q}^{\frac{n-1}{2}}_p(T) - \gamma_n\mathcal{Q}^{\frac{n-1}{2}}_s(T)F_n^L(p,T)\right)p\right. -\\
\left.-\bar{s}\left( F^R_n(s,T) \gamma_n \mathcal{Q}^{\frac{n-1}{2}}_p(T) - \gamma_n\mathcal{Q}^{\frac{n-1}{2}}_s(T)F_n^L(p,T)\right) \right](p^2 - 2s_0p + |s|^2)^{-1}.
\end{multline*}
\end{theorem}

\begin{proof}
We proof the statement by showing that
\[
\begin{split}
F_n^{R}(s,T)F_n^{L}(p,T)(p^2 - 2s_0p + |s|^2) &= \Big( F^R_n(s,T) \gamma_n \mathcal{Q}^{\frac{n-1}{2}}_p(T)
 -\gamma_n\mathcal{Q}^{\frac{n-1}{2}}_s(T)F_n^L(p,T)\Big)p
 \\
 &
-\bar{s}\Big( F^R_n(s,T) \gamma_n \mathcal{Q}^{\frac{n-1}{2}}_p(T) - \gamma_n\mathcal{Q}^{\frac{n-1}{2}}_s(T)F_n^L(p,T)\Big).
\end{split}
\]
If we apply the generalized $F$-resolvent equations \eqref{GFREL} and \eqref{GFRER}, we obtain
\[
\begin{split}
F_n^R(s,T)&F_n^L(p,T) p^2
\\
&
= F_n^R(s,T)[T^2F_n^L(p,T) + T\gamma_n\mathcal{Q}_p^{\frac{n-1}{2}}(T) + \gamma_n\mathcal{Q}_p^{\frac{n-1}{2}}(T)p]
\\
&=
 F_n^R(s,T)T^2F_n^L(p,T) + F_n^R(s,T)T\gamma_n\mathcal{Q}_p^{\frac{n-1}{2}}(T) + F_n^R(s,T)\gamma_n\mathcal{Q}_p^{\frac{n-1}{2}}(T)p
 \\
&=
[s^2F_n^R(s,T) -\gamma_n \mathcal{Q}_s^{\frac{n-1}{2}}(T)T - s\gamma_n\mathcal{Q}_s^{\frac{n-1}{2}}(T) ]F_n^L(p,T)
\\
&
+ [sF_n^R(s,T) -\gamma_n \mathcal{Q}_s^{\frac{n-1}{2}}(T)]\gamma_n\mathcal{Q}_p^{\frac{n-1}{2}}(T) + F_n^R(s,T)\gamma_n\mathcal{Q}_p^{\frac{n-1}{2}}(T)p
\\
&
=s^2F_n^R(s,T)F_n^L(p,T) -\gamma_n \mathcal{Q}_s^{\frac{n-1}{2}}(T)TF_n^L(p,T) - s\gamma_n\mathcal{Q}_s^{\frac{n-1}{2}}(T) F_n^L(p,T)
\\
&
+ sF_n^R(s,T)\gamma_n\mathcal{Q}_p^{\frac{n-1}{2}}(T) -\gamma_n^2 \mathcal{Q}_s^{\frac{n-1}{2}}(T)\mathcal{Q}_p^{\frac{n-1}{2}}(T) + F_n^R(s,T)\gamma_n\mathcal{Q}_p^{\frac{n-1}{2}}(T)p.
\end{split}
\]
In the a similar way, by applying the $F$-resolvent equations \eqref{FREL} and \eqref{FRER}, we obtain
\[
\begin{split}
F_n^R(s,T)F_n^L(p,T)2s_0p &= 2s_0F_n^R(s,T)[TF_n^L(p,T) +\gamma_n\mathcal{Q}_p^{\frac{n-1}{2}}(T)]
 \\
&
= 2s_0[sF_n^R(s,T) - \gamma_n\mathcal{Q}_s^{\frac{n-1}{2}}(T) ]F_n^L(p,T) + 2s_0F_n^R(s,T)\gamma_n\mathcal{Q}_p^{\frac{n-1}{2}}(T)
\\
&
= 2s_0sF_n^R(s,T)F_n^L(p,T) - 2s_0\gamma_n\mathcal{Q}_s^{\frac{n-1}{2}}(T)F_n^L(p,T)
\\
&+ 2s_0F_n^R(s,T)\gamma_n\mathcal{Q}_p^{\frac{n-1}{2}}(T).
\end{split}
\]
Therefore we have
\[
\begin{split}
F_n^R(s,T)F_n^L(p,T)&(p^2 - 2s_0s + \vert s\vert^2)
\\
&
= (s^2-2s_0s +\vert s \vert^2)F_n^R(s,T)F_n^L(p,T) -\gamma_n\mathcal{Q}_s^{\frac{n-1}{2}}(T)TF_n^L(p,T)
\\
&
+ (2s_0-s)\gamma_n\mathcal{Q}_s^{\frac{n-1}{2}}(T)F_n^L(p,T) + (s-2s_0)F_n^R(s,T)\gamma_nQ_p^{\frac{n-1}{2}}(T)
\\
&
-\gamma_n^2\mathcal{Q}_s^{\frac{n-1}{2}}(T)\mathcal{Q}_p^{\frac{n-1}{2}}(T) + F_n^R(s,T)\gamma_n\mathcal{Q}_p^{\frac{n-1}{2}}(T)p.
\end{split}
\]
As $2s_0 - s = \bar{s}$ and $s^2 - 2s_0s+ \vert s\vert^2 = 0$, by applying the $F$-resolvent equation \eqref{FREL} once again to $\gamma_n\mathcal{Q}_s^{\frac{n-1}{2}}(T)TF_n^L(p,T)$, we get
\[
\begin{split}
F_n^R(s,T)F_n^L(p,T)&(p^2 - 2s_0s + \vert s\vert^2)
\\
&
= -\gamma_n\mathcal{Q}_s^{\frac{n-1}{2}}(T)[F_n^L(p,T)p - \gamma_n\mathcal{Q}_p^{\frac{n-1}{2}}(T)]
 \\
 &
-\bar{s}(F_n^R(s,T)\gamma_nQ_p^{\frac{n-1}{2}}(T) - \gamma_n\mathcal{Q}_s^{\frac{n-1}{2}}(T)F_n^L(p,T)) \\
&
-\gamma_n^2\mathcal{Q}_s^{\frac{n-1}{2}}(T)\mathcal{Q}_p^{\frac{n-1}{2}}(T) + F_n^R(s,T)\gamma_n\mathcal{Q}_p^{\frac{n-1}{2}}(T)p
\\
&
=[F_n^R(s,T)\gamma_nQ_p^{\frac{n-1}{2}}(T) - \gamma_n\mathcal{Q}_s^{\frac{n-1}{2}}(T)F_n^L(p,T)]p
 \\
 &
 - \bar{s}[F_n^R(s,T)\gamma_nQ_p^{\frac{n-1}{2}}(T) - \gamma_n\mathcal{Q}_s^{\frac{n-1}{2}}(T)F_n^L(p,T)].
\end{split}
\]

\end{proof}

 The pseudo $F$-resolvent equation can be written in terms of the $F$-resolvent operators only using the left and the right $F$-resolvent equations.
\begin{corol}[The pseudo $F$-resolvent equation, form II]
Let $n$ be an odd number and let $T\in\mathcal{BC}^{0,1}(V_n)$.
Then for $p,s\in \rho_F(T)$ the following equation holds
\begin{multline*}
F_n^{R}(s,T)F_n^{L}(p,T)
\\
= \left[ \left( F^R_n(s,T)(F_n^L(p,T)p-TF_n^L(p,T)) - (sF_n^R(s,T)-F_n^R(s,T)T)F_n^L(p,T)\right)p\right. -\\
\left.-\bar{s}\left( F^R_n(s,T)(F_n^L(p,T)p-TF_n^L(p,T)) - (sF_n^R(s,T)-F_n^R(s,T)T)F_n^L(p,T)\right) \right](p^2 - 2s_0p + |s|^2)^{-1}.
\end{multline*}
\end{corol}
\begin{proof}
It is a direct consequence of Theorem \ref{THPSEUDO} and of Theorem \ref{FRE}.

\end{proof}

\begin{remark}{\rm
We conclude this section with an important property of the operators $\breve{P}_j$, defined in (\ref{pigei}),
 that in dimension $n=3$ are the Riesz projectors associated to a given
 paravector operator $T$ with commuting components as it is proved in Section 7.
 The case $n>3$ is still under investigation and it is related to the structure of the $F$-resolvent equation
 in dimension $n>3$.
}
\end{remark}
We begin by recalling the definition of projectors and
some of their basic properties.

\begin{definition}
Let $V$ be a  Banach module. We say that $P$ is a projector if $P^2=P$.
\end{definition}

For readability purposes, we put the proof of the following lemma in the last section of this paper.

\begin{lemma}\label{Pn=gamman}
Let $n$ be an odd number and let $\mathcal{P}_{n-1,n} = \Delta^{\frac{n-1}{2}}x^{n-1}$ be the monogenic polynomial defined on $\mathbb{R}^{n+1}$. Then we have $\mathcal{P}_{n-1,n} \equiv \gamma_n$.
\end{lemma}

Lemma \ref{Pn=gamman} motivates the definition of the operators $\breve{P}_j$.

\begin{theorem}\label{PTcomFCAL}
Let $n$ be an odd number, $T\in\mathcal{BC}^{\small 0,1}(V_n)$, $f\in\mathcal{SM}^L_{\sigma_F(T)}$.
Let $\sigma_F(T)= \sigma_{1F}\cup \sigma_{2F}$,
with ${\rm dist}\,( \sigma_{1F},\sigma_{2F})>0$. Let $U_1$ and
$U_2$ be two  admissible sets for $T$ such that  $\sigma_{1F} \subset U_1$ and $
\sigma_{2F}\subset U_2$,  with $\overline{U}_1
\cap\overline{U}_2=\emptyset$. Set
\begin{align}\label{pigei}
&\breve{P}_j:=\frac{\gamma_{n}^{-1}}{2\pi }\int_{\partial (U_j\cap \mathbb{C}_I)}F_n^L(s,T) \,
ds_Is^{n-1}, \ \ \ \ \ j=1,2,\\
\label{tigei}
&\breve{T}_j:=\frac{\gamma_{n}^{-1}}{2\pi }\int_{\partial (U_j\cap \mathbb{C}_I)}F_n^L(s,T)
ds_I\,s^{n}-\frac{1}{2\pi}\int_{\partial(U_j\cap\mathbb{C}_I)}\mathcal{Q}_s(T)^{\frac{n-1}{2}}ds_I s^{n-1} , \ \ \ \  j=1,2.
\end{align}
Then the following properties hold:
\begin{itemize}
\item[(1)]\label{PTcomFCAL1}
$T\breve{P}_j=\breve{P}_jT = T_j$ for $j=1,2$.
\item[(2)]For $\lambda\in\rho_F(T)$ we have
\begin{align}
\label{FREL_subs}\breve{P}_jF_n^{L}(\lambda,T)\lambda -\breve{T}_jF_n^{L}(\lambda,T) &= \breve{P}_j\gamma_n\mathcal{Q}_\lambda(T)^{\frac{n-1}{2}}\qquad j = 1,2.\\
\label{FRER_subs}\lambda F_n^{R}(\lambda,T) \breve{P}_j - F_n^{R}(\lambda,T)\breve{T}_j &=\gamma_n\mathcal{Q}_\lambda(T)^{\frac{n-1}{2}} \breve{P}_j\qquad j = 1,2.
\end{align}
\end{itemize}
\end{theorem}
\begin{proof}
Note, that the operators $\breve{P}_j$ and $\breve{T}_j$ can also be written using the right $F$-resolvent operator. To proof (1), we apply the $F$-resolvent equation and get
\begin{align*}
T\breve{P}_j &= T\frac{\gamma_{n}^{-1}}{2\pi}\int_{\partial(U_j\cap\mathbb{C}_I)}F_n^L(s,T)ds_Is^{n-1}
\\
&= \frac{\gamma_{n}^{-1}}{2\pi}\int_{\partial(U_j\cap\mathbb{C}_I)}TF_n^L(s,T)ds_Is^{n-1}
\\
&= \frac{\gamma_{n}^{-1}}{2\pi}\int_{\partial(U_j\cap\mathbb{C}_I)} \left(F_n^L(s,T)s - \gamma_n\mathcal{Q}_s(T)^{\frac{n-1}{2}} \right)ds_I s^{n-1}
\\
&
= \frac{\gamma_{n}^{-1}}{2\pi}\int_{\partial(U_j\cap\mathbb{C}_I)} F_n^L(s,T)ds_I s^{n} - \frac{1}{2\pi}\int_{\partial(U_j\cap\mathbb{C}_I)} \mathcal{Q}_s(T)^{\frac{n-1}{2}} ds_I s^{n-1}
\\
&= \breve{T}_j.
\end{align*}
As $s$, $ds_I$ and $\mathcal{Q}_s(T)$ commute on $\mathbb{C}_I$, we also have
\begin{align*}
\breve{P}_j T &= \frac{\gamma_{n}^{-1}}{2\pi}\int_{\partial(U_j\cap\mathbb{C}_I}s^{n-1}ds_IF_n^{R}(s,T)T
\\
&=\frac{\gamma_{n}^{-1}}{2\pi}\int_{\partial(U_j\cap\mathbb{C}_I}s^{n-1}ds_I\left(sF_n^{R}(s,T) - \gamma_n\mathcal{Q}_s(T)^{\frac{n-1}{s}}\right)
\\
& = \frac{\gamma_{n}^{-1}}{2\pi}\int_{\partial(U_j\cap\mathbb{C}_I}s^{n}ds_IF_n^{R}(s,T) -\frac{1}{2\pi}\int_{\partial(U_j\cap\mathbb{C}_I)} s^{n-1}ds_I\mathcal{Q}_s(T)^{\frac{n-1}{s}}
\\
&= \breve{T}_j.
\end{align*}

To proof \eqref{FREL_subs}, for $\lambda\in\rho_F(T)$ we apply the $F$-resolvent equation \eqref{FREL} and obtain
\[
\begin{split}
\breve{P}_jF_n^L(\lambda,T)\lambda &= \breve{P}_j\left(TF_n^L(\lambda,T)-\gamma_n\mathcal{Q}_{\lambda}(T)^{\frac{n-1}{2}}\right)
\\
&
=\breve{P}_jTF_n^L(\lambda,T) - \breve{P_j}\gamma_n\mathcal{Q}_{\lambda}(T)^{\frac{n-1}{2}}
 \\
 &
 = \breve{T}_jF_n^L(\lambda,T) - \breve{P}_j\gamma_n\mathcal{Q}_{\lambda}(T)^{\frac{n-1}{2}}.
\end{split}
\]
The identity \eqref{FRER_subs} can be proved analogously.

\end{proof}

\section{Preliminary results on the $SC$-functional calculus}

As we have already mentioned in the introduction the $F$-resolvent equation for the $F$-functional calculus involves also the $SC$-resolvent operators. In this section we recall some  results on the $SC$-functional calculus,
 for more details see \cite{SlaisComm}.

\begin{definition}(The $SC$-resolvent operators)
Let
$T\in\mathcal{BC}^{\small 0,1}(V_n)$
 and $ s \in \rho_F(T)$.
We define the left $SC$-resolvent operator as
\begin{equation}\label{SCCresolvoperator}
S_{C,L}^{-1}(s,T):=(s\mathcal{I}- \overline{T})(s^2\mathcal{I}-s(T+\overline{T})+T\overline{T})^{-1},
\end{equation}
and  the right $SC$-resolvent operator as
\begin{equation}\label{SCCresolvoperatorRQ}
S_{C,R}^{-1}(s,T):=(s^2\mathcal{I}-s(T+\overline{T})+T\overline{T})^{-1}(s\mathcal{I}- \overline{T}).
\end{equation}
\end{definition}

\begin{theorem}
Let $T\in\mathcal{BC}^{\small 0,1}(V_n)$  and $s\in \rho_{F}(T)$.  Then
$S_{C,L}^{-1}(s,T)$ satisfies the left $SC$-resolvent equation
\begin{equation}\label{Sresolveqbound}
S_{C,L}^{-1}(s,T)s-TS_{C,L}^{-1}(s,T)=\mathcal{I},
\end{equation}
and $S_{C,R}^{-1}(s,T)$ satisfies the right $SC$-resolvent equation
$$
sS_{C, R}^{-1}(s,T)-S_{C, R}^{-1}(s,T)T=\mathcal{I}.
$$
\end{theorem}
The following crucial results is proved in \cite{acgs}.
\begin{theorem}[The $SC$-resolvent equation]\label{RLRESOLVEQSCCALSC}
Let $T\in\mathcal{BC}^{\small 0,1}(V_n)$ and let $s$ and $p\in \rho_F(T)$. Then we have
\begin{equation}\label{RLresolvSC}
\begin{split}
S_{C,R}^{-1}(s,T)&S_{C,L}^{-1}(p,T)=((S_{C,R}^{-1}(s,T)-S_{C,L}^{-1}(p,T))p
\\
&
-\overline{s}(S_{C,R}^{-1}(s,T)-S_{C,L}^{-1}(p,T)))(p^2-2s_0p+|s|^2)^{-1}.
\end{split}
\end{equation}
Moreover, the resolvent equation can also be written as
\begin{equation}\label{RLresolvIISC}
\begin{split}
S_{C,R}^{-1}(s,T)&S_{C,L}^{-1}(p,T)=(s^2-2p_0s+|p|^2)^{-1}(s(S_{C,R}^{-1}(s,T)-S_{C,L}^{-1}(p,T))
\\
&
-(S_{C,R}^{-1}(s,T)-S_{C,L}^{-1}(p,T))\overline{p} ).
\end{split}
\end{equation}
\end{theorem}

\begin{definition}[The $SC$-functional calculus]\label{fdiTSC}
Let $T\in\mathcal{BC}^{\small 0,1}(V_n)$.
Let $U\subset \mathbb{R}^{n+1}$ be admissible for $T$
and set $ds_I=ds/I$ for $I\in\mathbb{S}$.
We define
\begin{equation}\label{FCSC}
f(T)= {{1}\over{2\pi }} \int_{\partial (U\cap \mathbb{C}_I) } S_{C,L}^{-1} (s,T)\  ds_I \
f(s),  \  \ {\rm for} \ \ f\in \mathcal{SM}_{\sigma_F(T)},
\end{equation}
and
\begin{equation}\label{FCSCRright}
f(T)= {{1}\over{2\pi }} \int_{\partial (U\cap \mathbb{C}_I) } f(s)\  ds_I \
S_{C,R}^{-1} (s,T), \  \ {\rm for} \ \ f\in \mathcal{SM}^R_{\sigma_F(T)}.
\end{equation}
\end{definition}
Finally we need a technical lemma.
\begin{definition}\label{defdiN}
Let $f:U\to \mathbb{R}_{n}$ be a slice monogenic function, where $U$ is an open set in $\mathbb{R}^{n+1}$.
 We define
$$
 \mathcal{N}(U)=\{ f\in\mathcal{SM}(U)\ :  \ f(U\cap \mathbb{C}_I)\subseteq  \mathbb{C}_I,\ \  \forall I\in \mathbb{S}\}.
$$
\end{definition}
First of all let us observe that  functions in the subclass  $\mathcal{N}(U)$ are both left and right slice hyperholomorphic. When we take the power series expansion of this class of functions at a point on the real line the coefficients of the expansion  are real numbers.

Now observe that for functions in $f\in \mathcal{N}(U)$ we can define $f(T)$
 using the left but also the right $SC$-functional calculus. It is
\[
\begin{split}
 f(T)&={{1}\over{2\pi }} \int_{\partial (U\cap \mathbb{C}_I)} S_{C,L}^{-1} (s,T)\  ds_I \ f(s)
 \\
 &
={{1}\over{2\pi }} \int_{\partial (U\cap \mathbb{C}_I)} \  f(s)\ ds_I \ S_{C,R}^{-1} (s,T).
\end{split}
\]
The following lemma is proved in \cite{acgs}.
\begin{lemma}\label{Lemma321} Let $B\in \mathcal{B}(V_n)$. Let $G$ be an axially symmetric s-domain and assume that
$f\in \mathcal{N}(G)$.
Then, for $p\in G$, we have
$$
\frac{1}{2\pi}\int_{\partial(G\cap\mathbb{C}_I)}f(s)ds_I
(\overline{s}B-Bp)(p^2-2s_0p+|s|^2)^{-1}=Bf(p).
$$
\end{lemma}

\section{Projectors for the dimension $n=3$}

To proof that operators $\breve{P}_j$ defined in (\ref{pigei}) are projectors is
based on a suitable $F$-resolvent equation that establishes a link between the product
$F_n^R(s,T)F_n^L(p,T)$ and the difference $F_n^R(s,T) - F_n^L(p,T)$.
For the case $n=3$ we are able to show that such a relation exists and we can prove that the operators
$\breve{P}_j$ are projectors.

\subsection{The $F$-resolvent equations for $n=3$}
We start with a preliminary lemma.
\begin{lemma}
Let $T\in\mathcal{BC}(V_3)$. Then for $p,s\in \rho_F(T)$ the following equation holds
\begin{multline*}
 F_3^R(s,T)S_{C,L}^{-1}(p,T)+S_{C,R}^{-1}(s,T)F_3^L(p,T)+\gamma_3\mathcal{Q}_s(T)\mathcal{Q}_p(T)
 \\
 = \left[ \left( F_3^R(s,T) - F_3^L(p,T)\right)p-\bar{s}\left( F_3^R(s,T) - F_3^L(p,T)\right) \right](p^2 - 2s_0p + |s|^2)^{-1}.
\end{multline*}
\end{lemma}
\begin{proof}
Let us consider the $SC$-resolvent equation and multiply it on the left by  $\gamma_3 \mathcal{Q}_s(T)$, we get
\begin{multline*}
 F_3^R(s,T)S_{C,L}^{-1}(p,T)
 \\
 = \left[ \left( F_3^R(s,T) - \gamma_3 \mathcal{Q}_s(T)S^{-1}_L(p,T)\right)p-\bar{s}\left( F_3^R(s,T) - \gamma_3 \mathcal{Q}_s(T)S^{-1}_L(p,T)\right) \right](p^2 - 2s_0p + |s|^2)^{-1}.
\end{multline*}
Now multiply the $S$-resolvent equation  on the right by  $\gamma_3 \mathcal{Q}_p(T)$, we get
\begin{multline*}
 S_{C,R}^{-1}(s,T)F_3^L(p,T)
 \\
 = \left[ \left(S_{C,R}^{-1}(s,T) \gamma_3 \mathcal{Q}_p(T)-  F_3^L(p,T)\right)p-\bar{s}\left(S_{C,R}^{-1}(s,T) \gamma_3 \mathcal{Q}_p(T)-  F_3^L(p,T) \right) \right](p^2 - 2s_0p + |s|^2)^{-1}.
\end{multline*}
Now we add the above two equations to get

\begin{multline*}
 F_3^R(s,T)S_{C,L}^{-1}(p,T)+ S_{C,R}^{-1}(s,T)F_3^L(p,T)
 \\
 =
 \left[ \left( F_3^R(s,T) - F_3^L(p,T)\right)p-\bar{s}\left( F_3^R(s,T) - F_3^L(p,T)\right) \right](p^2 - 2s_0p + |s|^2)^{-1}
\\
 +\Big[ \left( S_{C,R}^{-1}(s,T) \gamma_3 \mathcal{Q}_p(T) - \gamma_3 \mathcal{Q}_s(T)S^{-1}_{C,L}(p,T)\right)p
\\ -\bar{s}\Big( S_{C,R}^{-1}(s,T) \gamma_3 \mathcal{Q}_p(T)
  - \gamma_3 \mathcal{Q}_s(T)S^{-1}_{C,L}(p,T)\Big) \Big](p^2 - 2s_0p + |s|^2)^{-1}.
 \end{multline*}
 Finally, we verify that
 $$
 \Big[ \left( S_{C,R}^{-1}(s,T) \gamma_3 \mathcal{Q}_p(T) - \gamma_3 \mathcal{Q}_s(T)S^{-1}_{C,L}(p,T)\right)p
 -\bar{s}\Big( S_{C,R}^{-1}(s,T) \gamma_3 \mathcal{Q}_p(T)
  - \gamma_3 \mathcal{Q}_s(T)S^{-1}_{C,L}(p,T)\Big) \Big]
  $$
  $$
  \times(p^2 - 2s_0p + |s|^2)^{-1}=-\gamma_3 \mathcal{Q}_s(T)\mathcal{Q}_p(T).
 $$
 This follows from
 $$
  \left( S_{C,R}^{-1}(s,T) \gamma_3 \mathcal{Q}_p(T) - \gamma_3 \mathcal{Q}_s(T)S^{-1}_{C,L}(p,T)\right)p
 -\bar{s}\Big( S_{C,R}^{-1}(s,T) \gamma_3 \mathcal{Q}_p(T)
  - \gamma_3 \mathcal{Q}_s(T)S^{-1}_{C,L}(p,T)\Big)
  $$
  $$
  = \gamma_3\Big[ \Big( \mathcal{Q}_s(T) (s\mathcal{I}-\overline{T})
  \mathcal{Q}_p(T) -  \mathcal{Q}_s(T)(p\mathcal{I}-\overline{T})\mathcal{Q}_p(T) \Big)p
  $$
  $$
 -\bar{s}\Big( \mathcal{Q}_s(T) (s\mathcal{I}-\overline{T})
  \mathcal{Q}_p(T) -  \mathcal{Q}_s(T)(p\mathcal{I}-\overline{T})\mathcal{Q}_p(T)\Big) \Big]
  $$
  $$
  = \gamma_3\Big[ \mathcal{Q}_s(T)(  s -  p)\mathcal{Q}_p(T) p
 -\bar{s}\mathcal{Q}_s(T)(  s -  p)\mathcal{Q}_p(T)\Big]
  $$
   $$
  = \gamma_3\Big[ \mathcal{Q}_s(T)(  sp -  p^2)\mathcal{Q}_p(T)
 -\mathcal{Q}_s(T)(  \bar{s}s -  \bar{s}p)\mathcal{Q}_p(T)\Big]
  $$$$
  = \gamma_3\Big[ \mathcal{Q}_s(T)(  sp -  p^2)\mathcal{Q}_p(T)
 -\mathcal{Q}_s(T)(  \bar{s}s -  \bar{s}p)\mathcal{Q}_p(T)\Big]
  $$
  $$
  =- \gamma_3 \mathcal{Q}_s(T)\mathcal{Q}_p(T)(p^2 - 2s_0p + |s|^2).
  $$
\end{proof}
\begin{theorem}[The $F$-resolvent equation for $n=3$]
Let $T\in\mathcal{BC}(V_3)$. Then for $p,s\in \rho_F(T)$ the following equation holds
\begin{multline*}
 F_3^R(s,T)S_{C,L}^{-1}(p,T)+S_{C,R}^{-1}(s,T)F_3^L(p,T)+\gamma_3^{-1}\Big( sF_3^R(s,T)F_3^L(p,T)p
 \\
 -sF_3^R(s,T)TF_3^L(p,T) -F_3^R(s,T)TF_3^L(p,T)p+F_3^R(s,T)T^2F_3^L(p,T)\Big)
 \\
 = \left[ \left( F_3^R(s,T) - F_3^L(p,T)\right)p-\bar{s}\left( F_3^R(s,T) - F_3^L(p,T)\right) \right](p^2 - 2s_0p + |s|^2)^{-1}.
\end{multline*}
\end{theorem}
\begin{proof}
We now replace in the term $\mathcal{Q}_s(T)\mathcal{Q}_p(T)$ the right and the left $F$-resolvent equations written for $p,s\in \rho_F(T)$ as:
\begin{equation}
sF_3^R(s,T)-F_3^R(s,T)T=\gamma_3\mathcal{Q}_s(T),
\end{equation}
$$
F_3^L(p,T)p-TF_3^L(p,T)=\gamma_3\mathcal{Q}_p(T).
$$
So we get
$$
\gamma_3^2\mathcal{Q}_s(T)\mathcal{Q}_p(T)=(sF_3^R(s,T)-F_3^R(s,T)T)(F_3^L(p,T)p-TF_3^L(p,T))
$$
$$
=sF_3^R(s,T)F_3^L(p,T)p-sF_3^R(s,T)TF_3^L(p,T) -F_3^R(s,T)TF_3^L(p,T)p+F_3^R(s,T)T^2F_3^L(p,T).
$$
\end{proof}
In the sequel we will need this lemma which is based on the monogenic functional calculus, see the book \cite{jefferies} for more details
 (or some of the papers \cite{jmc,jmcpw,mcp} where the calculus was introduced).

\begin{lemma}\label{lemmong}
Let $T\in\mathcal{BC}(V_3)$.  Suppose that $G$ contains just some points of the $F$-spectrum of $T$ and assume that
the closed smooth curve $\partial (G\cap \mathbb{C}_I)$ belongs to the $F$-resolvent set of $T$, for every $I\in \mathbb{S}$.
Then
\begin{equation*}
 \int_{\partial (G\cap \mathbb{C}_I)}ds_I\, s\, F_{3}^R(s,T)= 0\quad\text{and}\quad
 \int_{\partial (G\cap \mathbb{C}_I)}F_{3}^L(p,T)\, p\, dp_I = 0.
\end{equation*}
 \end{lemma}
\begin{proof}
As $\Delta x \equiv 0$, we have
\begin{equation*}
\int_{\partial (G\cap \mathbb{C}_I)}ds_I\, s\, F_{3}^R(s,x) =0 \quad\text{and}\quad \int_{\partial (G\cap \mathbb{C}_I)}F_{3}^L(p,x)\, p\, dp_I = 0
\end{equation*}
for all $x$ such that $x\notin[s]$ if $s\in \partial (G\cap \mathbb{C}_I)$ resp. for all $x$ such that $x\notin[p]$ if $p\in\partial (G\cap \mathbb{C}_I)$.
We consider the case of $F_{3}^L(p,x)$, the other case can be treated in a similar way.
We now recall that $F_{3}^L(p,x)$ is left monogenic in $x$ for every $p$, such that $x\notin[p]$.
Therefore,  using the monogenic functional calculus, see \cite{jefferies}, we write
$$
F_{3}^L(p,T)=\int_{\partial\Omega} \mathcal{G}_\omega(T) \mathbf{n}(\omega) F_{3}^L(p,\omega)d\mu(\omega)
$$
where the open set $\Omega$ contains the monogenic spectrum of $T$, $\mathcal{G}_\omega(T)$ is the monogenic resolvent operator, $\mathbf{n}(\omega)$ is the unit normal vector  to $\partial\Omega$ and  $d\mu(\omega)$ is the surface element.
Using  Fubini's theorem we have:
\[
\begin{split}
\int_{\partial (G\cap \mathbb{C}_I)}F_{3}^L(p,T)pdp_I
&=\int_{\partial (G\cap \mathbb{C}_I)}\int_{\partial\Omega} \Big(\mathcal{G}_\omega(T)\mathbf{n}(\omega)F_{3}^L(p,\omega)d\mu(\omega)\Big) pdp_I
\\
&
=\int_{\partial\Omega} \mathcal{G}_\omega(T)\mathbf{n}(\omega)
\Big(\int_{\partial (G\cap \mathbb{C}_I)}F_{3}^L(p,\omega) pdp_I\Big)d\mu(\omega)
\\
&
=0,
\end{split}
\]
which concludes the proof.
\end{proof}

\begin{theorem}Let $T\in\mathcal{BC}^{0,1}(V_{3})$ and let $\sigma_F(T) = \sigma_{F,1}(T)\cup\sigma_{F,2}(T)$ with
$$
\mathrm{dist}(\sigma_{F,1}(T),\sigma_{F,2}(T)) > 0.
$$
 Let $G_1,G_2\subset\mathbb{H}$ be two admissible sets for $T$ such that $\sigma_{F,1}(T)\subset G_1$ and $\overline{G_1}\subset G_2$ and such that $\mathrm{dist}(G_2,\sigma_{F,2}(T))>0$. Then the operator
\begin{equation*}
\breve{P}:=\frac{\gamma_3^{-1}}{2\pi }\int_{\partial (G_1\cap \mathbb{C}_I)}F_3^L(p,T) \,dp_Ip^{2} = \frac{\gamma_3^{-1}}{2\pi }\int_{\partial (G_2\cap \mathbb{C}_I)}s^{2}ds_IF_3^R(s,T)
\end{equation*}
is a projector, i.e. we have
$$\breve{P}^2=\breve{P}.$$
\end{theorem}
\begin{proof}

If we multiply the resolvent equation (\ref{RLresolvSC}) by $s$ on the left and by $p$ on the right, we get
\begin{multline*}
 sF_{3}^R(s,T)S_{C,L}^{-1}(p,T)p+sS_{C,R}^{-1}(s,T)F_{3}^L(p,T)p
 \\
 +\gamma_3^{-1}\Big( s^2F_{3}^R(s,T)F_{3}^L(p,T)p^2-s^2F_{3}^R(s,T)TF_{3}^L(p,T)p
 \\
  -sF_{3}^R(s,T)TF_{3}^L(p,T)p^2+sF_{3}^R(s,T)T^2F_{3}^L(p,T)p\Big)
 \\
 = s\left[ \left( F_{3}^R(s,T) - F_{3}^L(p,T)\right)p-\bar{s}\left( F_{3}^R(s,T) - F_{3}^L(p,T)\right) \right](p^2 - 2s_0p + |s|^2)^{-1}p.
\end{multline*}
If we multiply this equation by $ds_I$ on the left integrate it over $\partial(G_2\cap\mathbb{C}_I)$ with respect to $ds_I$ and  then multiply it by $dp_I$ on the right and integrate over $\partial(G_1\cap\mathbb{C}_I)$ with respect to $dp_I$, we obtain
\[
\begin{split}
\int_{\partial (G_2\cap \mathbb{C}_I)}ds_I s& F_{3}^R(s,T) \int_{\partial (G_1\cap \mathbb{C}_I)} S_{C,L}^{-1}(p,T)p\, dp_I
\\
&
+\int_{\partial (G_2\cap \mathbb{C}_I)}ds_Is S_{C,R}^{-1}(s,T) \int_{\partial (G_1\cap \mathbb{C}_I)} F_{3}^L(p,T)p\, dp_I
 \\
&
 +\gamma_3^{-1}\Big( \int_{\partial (G_2\cap \mathbb{C}_I)}ds_Is^2F_{3}^R(s,T) \int_{\partial (G_1\cap \mathbb{C}_I)}  F_{3}^L(p,T)p^2dp_I
 \\
&
 -\int_{\partial (G_2\cap \mathbb{C}_I)}ds_Is^2F_{3}^R(s,T)T \int_{\partial (G_1\cap \mathbb{C}_I)} F_{3}^L(p,T)p\, dp_I
\end{split}
\]

\[
\begin{split}
\ \ \ \ \ \ \ \ \ \ \ \ \ \ \  \ \ \  \ \  \ & -\int_{\partial (G_2\cap \mathbb{C}_I)}ds_IsF_{3}^R(s,T)T\int_{\partial (G_1\cap \mathbb{C}_I)} F_{3}^L(p,T)p^2\, dp_I
  \\
&
  +\int_{\partial (G_2\cap \mathbb{C}_I)}ds_IsF_{3}^R(s,T)T^2\int_{\partial (G_1\cap \mathbb{C}_I)} F_{3}^L(p,T)p \, dp_I\Big)
 \\
&
 = \int_{\partial (G_2\cap \mathbb{C}_I)}ds_I\int_{\partial (G_1\cap \mathbb{C}_I)}
 s\Big[ \left( F_{3}^R(s,T) - F_{3}^L(p,T)\right)p
 \\
&
 -\bar{s}\left( F_{3}^R(s,T) - F_{3}^L(p,T)\right) \Big](p^2 - 2s_0p + |s|^2)^{-1}p\,dp_I.
\end{split}
\]
Using Lemma \ref{lemmong} we have
\begin{multline*}
 \gamma_3^{-1}\int_{\partial (G_2\cap \mathbb{C}_I)}ds_Is^2F_{3}^R(s,T) \int_{\partial (G_1\cap \mathbb{C}_I)}  F_{3}^L(p,T)p^2dp_I
 \\
 = \int_{\partial (G_2\cap \mathbb{C}_I)}ds_I\int_{\partial (G_1\cap \mathbb{C}_I)}
 s\Big[ \left( F_{3}^R(s,T) - F_{3}^L(p,T)\right)p
 \\
 -\bar{s}\left( F_{3}^R(s,T) - F_{3}^L(p,T)\right) \Big](p^2 - 2s_0p + |s|^2)^{-1} \,dp_Ip.
\end{multline*}
This is equal to
\begin{multline}
\label{proof_proj_eq1}
\frac{(2\pi)^2}{\gamma_3^{-1}}\breve{P}^2 = \int_{\partial (G_2\cap \mathbb{C}_I)}ds_I\int_{\partial (G_1\cap \mathbb{C}_I)}
 s\Big[ \left( F_{3}^R(s,T) - F_{3}^L(p,T)\right)p
 \\
 -\bar{s}\left( F_{3}^R(s,T) - F_{3}^L(p,T)\right) \Big](p^2 - 2s_0p + |s|^2)^{-1} \,dp_Ip.
\end{multline}
Let us observe the integral on the right hand side. As $\overline{G_1}\subset G_2$, for any
$s\in\partial(G_2\cap\mathbb{C}_I)$ the functions
$$
p\mapsto p(p^2 -2s_0 p + \vert s\vert^2)^{-1}p
$$
 and
 $$
 p\mapsto (p^2 - 2s_0p+\vert s\vert^2)^{-1}p
 $$ are  slice monogenic on $\overline{G_1}$. Therefore, we have
\begin{equation*}
\int_{\partial(G_1\cap\mathbb{C}_I)}p(p^2 -2s_0 p + \vert s\vert^2)^{-1}p = 0\quad\text{and}\quad\int_{\partial(G_1\cap\mathbb{C}_I)}(p^2 -2s_0 p + \vert s\vert^2)^{-1}pdp_I = 0
\end{equation*}
and it follows that
$$
 \int_{\partial (G_2\cap \mathbb{C}_I)}ds_I\int_{\partial (G_1\cap \mathbb{C}_I)}
 sF_{3}^R(s,T) p(p^2 - 2s_0p + |s|^2)^{-1} \,dp_Ip=0
 $$
 and
 $$
 \int_{\partial (G_2\cap \mathbb{C}_I)}ds_I\int_{\partial (G_1\cap \mathbb{C}_I)}
 s\bar{s}F_{3}^R(s,T) (p^2 - 2s_0p + |s|^2)^{-1} \,dp_Ip=0.
$$
Thus, \eqref{proof_proj_eq1} simplifies to
\begin{multline*}
\breve{P}^2
 = \frac{\gamma_3^{-1}}{(2\pi)^2}\int_{\partial (G_2\cap \mathbb{C}_I)}s\,ds_I\int_{\partial (G_1\cap \mathbb{C}_I)}
 \Big[ \left(\bar{s} F_{3}^L(p,T) - F_{3}^L(p,T)p\right) \Big](p^2 - 2s_0p + |s|^2)^{-1} \,dp_Ip
\end{multline*}
and by applying Lemma (\ref{Lemma321}) we finally obtain
$$
\breve{P}^2 = \frac{\gamma_3^{-1}}{2\pi}\int_{\partial(G_1\cap\mathbb{C}_I)}F_3^L(p,T)p\,dp_I p  = \breve{P}.
$$

\end{proof}

\section{Formulations of the quaternionic $F$-functional calculus }

We point out that, even thought the $F$-resolvent equation is known only for $n=3$,
 this case is of particular importance because it also allows to study
  the quaternionic version of the $F$-functional calculus.
In this section we will state the main results related with the quaternionic $F$-functional calculus
without details  since they very similar to the Clifford setting for $n=3$.

We denote by $\hh$ the algebra of quaternions.
The imaginary units  in $\hh$ are denoted by $i$, $j$ and $k$, respectively,
 and an element in $\hh$ is of the form $q=x_0+ix_1+jx_2+kx_3$, for $x_\ell\in\rr$, $\ell=0,1,2,3.$
The real part, the imaginary part and the modulus of a quaternion are defined as
${\rm Re}(q)=x_0$, $\underline{q}={\rm Im}(q)=i x_1 +j x_2 +k x_3$, $|q|^2=x_0^2+x_1^2+x_2^2+x_3^2.$
The conjugate of the quaternion $q$ is defined by $\bar q={\rm Re }(q)-{\rm Im }(q)=x_0-i x_1-j x_2-k x_3$
 and it satisfies $|q|^2=q\bar q=\bar q q.$
Let us denote by $\mathbb{S}$ the unit sphere of purely imaginary \index{$\mathbb{S}$}
quaternions, i.e.
$
\mathbb{S}=\{q=ix_1+jx_2+kx_3\ {\rm such \ that}\
x_1^2+x_2^2+x_3^2=1\}.
$
The Fueter mapping theorem consists in applying the Laplace operator in dimension $4$ to  functions of the form
$$
f(q)=\alpha(x_0,|\underline{q}|)+I\beta(x_0,|\underline{q}|)
$$
 where $\alpha$ and $\beta$ are suitable functions satisfying the Cauchy-Riemann system and $q=x_0+\underline{q}$ is a quaternion.  Functions of this form are slice regular, see the book \cite{bookfunctional} for more details.
\begin{definition}[Slice regular functions]\label{regularity} Let $U$ be an open set in
$\mathbb{H}$ and consider  a real differentiable function $f:U \to \mathbb{H}$.
Denote by  $f_I$ the restriction of $f$  to the complex plane $\mathbb{C}_I=\mathbb{R}+I\mathbb{R}$.
\begin{itemize}
\item
We say that  $f$  is (left) slice regular  if, for every $I \in
\mathbb{S}$, on $U \cap \mathbb{C}_I$ we have:
$$\frac{1}{2}\left(\frac{\partial}{\partial x}
+I\frac{\partial}{\partial y}\right)f_I(x+Iy)=0.$$
\\
The set of left slice regular functions on the open set $U$ is denoted by $\mathcal{SR}^L(U)$.
\item
We say that  $f$ is
right slice regular  if, for every $I \in
\mathbb{S}$, on $U \cap \mathbb{C}_I$ we have:
$$\frac{1}{2}\left(\frac{\partial}{\partial x}f_I(x+Iy)
+\frac{\partial}{\partial y}f_I(x+Iy)I\right)=0.$$
\\
The set of right slice regular functions on the open set $U$ is denoted by $\mathcal{SR}^R(U)$.
\end{itemize}
\end{definition}

\begin{definition}[Fueter regular functions] Let $U$ be an open set in $\mathbb{H}$.
A real differentiable function $f: U\to \mathbb{H}$ is left Cauchy-Fueter (for brevity just Fueter) regular if
$$
\frac{\pp}{\pp x_0}f(q)+ i\frac{\pp}{\pp x_1}f(q)+ j\frac{\pp}{\pp x_2}f(q)+ k\frac{\pp}{\pp x_3}f(q)=0,\ \ \ q\in U.
$$
It is right Fueter regular if
$$
\frac{\pp}{\pp x_0}f(q)+ \frac{\pp}{\pp x_1}f(q)i+ \frac{\pp}{\pp x_2}f(q)j+ \frac{\pp}{\pp x_3}f(q)k=0,\ \ \ q\in U.
$$
\end{definition}

\begin{definition}[The $F$-kernel]
Let $q$, $s\in \mathbb{H}$.
We define, for $s\not\in[q]$, the $F^L$-kernel as
$$
F^L(s,q):=-4(s-\bar q)(s^2-2{\rm Re}(q)s +|q|^2)^{-2},
$$
and the $F^R$-kernel as
$$
F^R(s,q):=-4(s^2-2{\rm Re}(q)s +|q|^2)^{-2}(s-\bar q).
$$
\end{definition}
With the above notation the Fueter mapping theorem in integral form becomes:
\begin{theorem}[The Fueter mapping theorem in integral form]
Set $ds_I=ds/ I$ and let $W\subset \mathbb{H}$ be an open set.
 Let $U$ be a bounded  axially symmetric s-domain such that $\overline{U}\subset W$.
 Suppose that the boundary of $U\cap \mathbb{C}_I$  consists of a
finite number of rectifiable Jordan curves for any $I\in\mathbb{S}$.
\begin{itemize}
\item[(a)]
If $q\in U$ and $f\in \mathcal{SR}^L(W)$ then
$\breve{f}(q)=\Delta f(q)$
is left Fueter regular and it admits the integral representation
\begin{equation}\label{Fueter}
 \breve{f}(q)=\frac{1}{2 \pi}\int_{\pp (U\cap \mathbb{C}_I)} F^L(s,q)ds_I f(s),
\end{equation}
\item[(b)]
If $q\in U$ and $f\in \mathcal{SR}^L(W)$ then
$\breve{f}(q)=\Delta f(q)$
is right Fueter regular and it admits the integral representation
\begin{equation}\label{Fueter}
 \breve{f}(q)=\frac{1}{2 \pi}\int_{\pp (U\cap \mathbb{C}_I)} f(s)ds_I F^R(s,q).
\end{equation}
\end{itemize}
The integrals  depend neither on $U$ and nor on the imaginary unit $I\in\mathbb{S}$.
\end{theorem}

We now consider the formulations of the $F$-functional calculus in the quaternionic setting for right linear quaternionic operators. The same formulation holds also for left linear operators with a suitable interpretation of the symbols.
\begin{definition} Let $V$ be a right vector space on $\mathbb{H}$.
A map $T: V\to V$ is said to be a right linear operator if
$$
 T(u+v)=T(u)+T(v),
\ \ \ \ \ \
T(us)=T(u)s,
$$
for all  $s\in\mathbb{H}$ and  all $u,v\in V$.
\end{definition}
In the sequel, we will consider only two sided vector spaces $V$, otherwise
the set of right linear operators
is not a (left or right) vector space. With this assumption,
the set End$(V)$ of right linear operators on
$V$ is both a left and a right vector space on $\hh$ with respect to the operations
$$(aT)(v):=aT(v), \ \ \ \
(Ta)(v):=T(av).
$$

\begin{definition} Let $V$ be a bilateral quaternionic Banach space.
We will denote by $\mathcal{B}(V)$ the bilateral Banach space of all right linear
bounded operators $T:V\to V$.
\\
 We will denote by
$\mathcal{BC}(V)$ the subclass of $\mathcal{B}(V)$ that consists of those quaternionic operators
 $T$ that can be written as
$
T=T_0+iT_1+jT_2+kT_3
$
where  the operators $T_\ell$, $\ell=0,1,2,3$ commute among themselves.
\end{definition}
It is easy to verify that $\mathcal{B}(V)$ is a Banach space endowed with its natural norm.

\begin{definition}[The $F$-spectrum and the $F$-resolvent sets]\label{seesspect_quat}
Let $T\in\mathcal{BC}(V)$.
We define the $F$-spectrum $\sigma_F(T)$ of $T$ as
$$
\sigma_{F}(T)=\{ s\in \hh\ \ :\ \ \ \ \ s^2\mathcal{I}-s(T+\overline{T})+T\overline{T}\ \ \
{\rm is\ not\  invertible}\}.
$$
The $S$-resolvent set $\rho_S(T)$ is defined as
$$
\rho_{F}(T)=\mathbb{H}\setminus\sigma_{F}(T).
$$
\end{definition}
\begin{theorem}[Structure of the $F$-spectrum]\label{strutturaK_quat}
 Let $T\in\mathcal{BC}(V)$ and let $p = p_0 +p_1I\in
  p_0 +p_1\mathbb{S}\subset \mathbb{H}\setminus\mathbb{R}$, such that $p\in \sigma_{ F}(T)$.
  Then all the elements of the sphere $p_0 +p_1\mathbb{S}$
 belong to $\sigma_{ F}(T)$.
\end{theorem}
\begin{theorem}[Compactness of $F$-spectrum]\label{compattezaDS_quat}
\par\noindent
Let
$T\in\mathcal{BC}(V)$. Then
the $F$-spectrum $\sigma_{F} (T)$  is a compact nonempty set.

\end{theorem}

\begin{definition}[$F$-resolvent operators]
Let
$T\in\mathcal{BC}(V)$. For $s\in \rho_F(T)$
we define the left $F$-resolvent operator as
$$
F^L(s,T):=-4(s\mathcal{I}-\overline{ T})(s^2\mathcal{I}-s(T+\overline{T})+T\overline{T})^{-2},
$$
and the right $F$-resolvent operator as
$$
F^R(s,T):=-4(s^2\mathcal{I}-s(T+\overline{T})+T\overline{T})^{-2}(s\mathcal{I}-\overline{ T}).
$$
\end{definition}
The definition of $T$-admissible set $U$ and of  locally  left (resp. right) slice regular functions  on the $F$-spectrum
$\sigma_{F}(T)$ can be obtained by rephrasing  Definition \ref{def3.9_mondsa}.
\\
We will denote by $\mathcal{SR}^L_{\sigma_{F}(T)}$ (resp. right $\mathcal{SR}^R_{\sigma_{F}(T)}$)  the set of locally left (resp. right) slice regular functions on $\sigma_{F}(T)$.

\begin{definition}[The quaternionic $F$-functional calculus for bounded operators]
Let $T\in\mathcal{BC}(V)$ and set $ds_I=ds/I$ .
   We define the formulations of the  quaternionic $F$-functional calculus as
\begin{equation}\label{DEFinteg311REG}
\breve{f}(T):=\frac{1}{2\pi}\int_{\pp(U\cap \mathbb{C}_I)} F^L(s,T) \, ds_I\, f(s),\ \ \ \ \ f\in \mathcal{SR}^L_{\sigma_{F}(T)},
\end{equation}
and
\begin{equation}\label{DEFinteg311RIGHTREG}
\breve{f}(T):=\frac{1}{2\pi}\int_{\pp(U\cap \mathbb{C}_I)} f(s) \, ds_I\, F^R(s,T),\ \ \ \ \ f\in \mathcal{SR}^R_{\sigma_{F}(T)},
\end{equation}
where $U$ is $T$-admissible.
\end{definition}

\begin{theorem}[The quaternionic $F$-resolvent equation]
Let $T\in\mathcal{BC}(V)$. Then for $p,s\in \rho_F(T)$ the following equation holds
\begin{multline*}
 F^R(s,T)S_{C,L}^{-1}(p,T)+S_{C,R}^{-1}(s,T)F^L(p,T)+\frac{1}{4}\Big( sF^R(s,T)F^L(p,T)p
 \\
 -sF^R(s,T)TF^L(p,T) -F^R(s,T)TF^L(p,T)p+F^R(s,T)T^2F^L(p,T)\Big)
 \\
 = \left[ \left( F^R(s,T) - F^L(p,T)\right)p-\bar{s}\left( F^R(s,T) - F^L(p,T)\right) \right](p^2 - 2s_0p + |s|^2)^{-1}.
\end{multline*}
where the quaternionic $SC$-resolvent operators are defined as
\begin{equation}
S_{C,L}^{-1}(s,T):=(s\mathcal{I}- \overline{T})(s^2\mathcal{I}-s(T+\overline{T})+T\overline{T})^{-1},
\end{equation}
and
\begin{equation}
S_{C,R}^{-1}(s,T):=(s^2\mathcal{I}-s(T+\overline{T})+T\overline{T})^{-1}(s\mathcal{I}- \overline{T}).
\end{equation}
\end{theorem}
As a consequence of the quaternionic $F$-resolvent equations we can study the Riesz projectors associated with the quaternionic $F$-functional calculus. We have, as a consequence of Theorem \ref{PTcomFCAL}, in this quaternionic version and of the quaternionic $F$-resolvent equation:

\begin{theorem}
Let $T\in\mathcal{BC}(V)$.
Let $\sigma_F(T)= \sigma_{1F}(T)\cup \sigma_{2F}(T)$,
with ${\rm dist}\,( \sigma_{1F}(T),\sigma_{2F}(T))>0$. Let $U_1$ and
$U_2$ be two $T$-admissible sets such that  $\sigma_{1F}(T) \subset U_1$ and $
\sigma_{2F}(T)\subset U_2$,  with $\overline{U}_1 \cap\overline{U}_2=\emptyset$. Set
$$
\breve{P}_j:=\frac{C}{2\pi }\int_{\partial (U_j\cap \mathbb{C}_I)}F^L(s,T) \, ds_Is^{2}, \ \ \ \ \ j=1,2,
$$
where $C:=\Delta q^{2}$.
Then, for $j=1,2$, the following properties hold:
\begin{itemize}
\item[(1)] $\breve{P}_j^2=\breve{P}_j,$
\item[(2)]$T\breve{P}_j=\breve{P}_jT $.
\end{itemize}
\end{theorem}

\section{A proof of Lemma \ref{Pn=gamman}}
For the proof, we need the following identitiy.
\begin{proposition}
Let $m\geq0$. Then the following identity holds
\begin{equation}
\label{binom_id}\sum_{k=0}^{j}(-1)^k\binom{m+k-1}{k}\binom{m+j}{m+k} = 1\qquad j=0,1,2,\ldots.\\
\end{equation}
\end{proposition}
\begin{proof}
Let
$$
\Lambda(j,k):=(-1)^k\binom{m+k-1}{k}\binom{m+j}{m+k}.
$$
It is easy to check that the following recurrence relation is satisfied
\begin{align*}
-\frac{m+j+1}{j+2}\Lambda(j,k) + \frac{m+j+1}{j+2}\Lambda(j+1,k) + \frac{m+j+1}{j+2}\Lambda(j,k+1)&\\
-\frac{m+2j+3}{j+2}\Lambda(j+1,k+1) + \Lambda(j+2,k+1)&=0
\end{align*}
for all $j,k\in\mathbb{Z}$. Note that $\Lambda(j,k) = 0$ if $k<0$ or $k>j$. Thus, by taking the sum over all $k\in\mathbb{Z}$, we obtain
\begin{align*}
-\frac{m+j+1}{j+2}\sum_{k=0}^{j}\Lambda(j,k) + \frac{m+j+1}{j+2}\sum_{k=0}^{j+1}\Lambda(j+1,k) + \frac{m+j+1}{j+2}\sum_{k=-1}^{j-1}\Lambda(j,k+1)&\\
-\frac{m+2j+3}{j+2}\sum_{k=-1}^{j}\Lambda(j+1,k+1)+\sum_{k=-1}^{j+1}\Lambda(j+2,k+1)&=0.
\end{align*}
If we define $S(j):=\sum_{k=0}^{j}(-1)^k\binom{m+k-1}{k}\binom{m+j}{m+k}$, this equation turns into
\begin{align*}
-\frac{m+j+1}{j+2}S(j) + \frac{m+j+1}{j+2}S(j+1) + \frac{m+j+1}{j+2}S(j)&\\
-\frac{m+2j+3}{j+2}S(j+1)+S(j+2)&=0.
\end{align*}
which simplifies to
\begin{equation*}
-S(j+1) + S(j+2) = 0.
\end{equation*}
Thus, as $S(0) = 1$, by induction, we obtain \eqref{binom_id}.

\end{proof}

\begin{proof}[{\bf Proof of Lemma \ref{Pn=gamman}}]
Let $x\in\mathbb{R}^{n+1}$ and let $U\subset\mathbb{R}^{n+1}$ be a ball centered at $0$ such that $x\in U$. By \eqref{Fueter}, for any $I\in\mathbb{S}$, we have
\begin{align*}
\mathcal{P}_{n-1,n}(x)& = \frac{1}{2\pi}\int_{\partial(U\cap\mathbb{C}_{I})} F_n^{L}(s,x)ds_I s^{n-1}=\\
&= \frac{1}{2\pi}\int_{\partial(U\cap\mathbb{C}_{I})} \gamma_n(s-\bar{x})(s^2 - 2\mathrm{Re}(x)s + \vert x\vert^2)^{-\frac{n+1}{2}}ds_I s^{n-1}.
\end{align*}
Let $I = I_x$ and $m=\frac{n-1}{2}$. Then $s$ and $x$ commute and by applying the Residue theorem we obtain
\begin{equation*}
\mathcal{P}_{n-1,n}(x) = \frac{-I\gamma_n}{2\pi}\int_{\partial(U\cup\mathbb{C}_I)} \frac{1}{(s-x)^{m+1}} \frac{1}{(s-\bar{x})^{m}}s^{2m} ds = \gamma_{n}(\mathrm{Res}_x(f) + \mathrm{Res}_{\bar{x}}(f)),
\end{equation*}
where $f(s) := \frac{1}{(s-x)^{m+1}} \frac{1}{(s-\bar{x})^{m}}s^{2m} $. It is easy to see, that
\begin{gather*}
\frac{\partial^k}{\partial s^k} s^{2m} = \frac{(2m)!}{(2m-k)!}s^{2m-k},\\
\frac{\partial^k}{\partial s^k}\frac{1}{(s-\bar{x})^m}=(-1)^k\frac{(m+k-1)!}{(m-1)!}\frac{1}{(s-\bar{x})^{m+k}},\\
\frac{\partial^k}{\partial s^k}\frac{1}{(s-x)^{m+1}}=(-1)^k\frac{(m+k)!}{m!}\frac{1}{(s-x)^{m+1+k}}.
\end{gather*}
Thus, we obtain
\begin{align*}
\mathrm{Res}_{x}(f) &= \frac{1}{m!}\lim_{s\to x}\frac{\partial^m}{\partial s^m}\left((s-x)^{m+1}f(s)\right) = \frac{1}{m!}\lim_{s\to x}\frac{\partial^m}{\partial s^m}\left(\frac{1}{(s-\bar{x})^{m}}s^{2m}\right)
\\
&= \frac{1}{m!}\lim_{s\to x}\sum_{k=0}^{m}\binom{m}{k}\left(\frac{\partial^k}{\partial s^k}\frac{1}{(s-\bar{x})^{m}}\right)\left(\frac{\partial^{m-k}}{\partial s^{m-k}} s^{2m} \right)
 \\
&= \frac{1}{m!}\lim_{s\to x}\sum_{k=0}^m \binom{m}{k}(-1)^k\frac{(m+k-1)!}{(m-1)!}\frac{1}{(s-\bar{x})^{m+k}}\frac{(2m)!}{(m+k)!}s^{m+k}
 \\
&= \sum_{k=0}^{m}(-1)^k\binom{2m}{m-k}\binom{m+k-1}{k}\frac{x^{m+k}}{(x-\bar{x})^{m+k}}
\\
&= \frac{1}{(x-\bar{x})^{2m}}\sum_{k=0}^{m}(-1)^k\binom{2m}{m-k}\binom{m+k-1}{k}x^{m+k}(x-\bar{x})^{m-k}.
\end{align*}
But as
\begin{equation*}
(x-\bar{x})^{m-k} = \sum_{j=0}{m-k}\binom{m-k}{j}x^j(-\bar{x})^{m-k-j},
\end{equation*}
we have
\begin{align*}
\mathrm{Res}_{x}(f) &= \frac{1}{(x-\bar{x})^{2m}}\sum_{k=0}^{m}\sum_{j=0}^{m-k}(-1)^k\binom{2m}{m-k}
\binom{m+k-1}{k}\binom{m-k}{j}x^{m+k+j}(-\bar{x})^{m-k-j}\\
&= \frac{1}{(x-\bar{x})^{2m}}\sum_{k=0}^{m}\sum_{j=k}^{m}(-1)^k
\binom{2m}{m-k}\binom{m+k-1}{k}\binom{m-k}{j-k}x^{m+j}(-\bar{x})^{m-j}
\\
&= \frac{1}{(x-\bar{x})^{2m}}\sum_{j=0}^{m}\sum_{k=0}^{j}(-1)^k\binom{2m}{m-k}
\binom{m+k-1}{k}\binom{m-k}{j-k}x^{m+j}(-\bar{x})^{m-j}.
\end{align*}
For the coefficients, we obtain
\begin{align*}
\sum_{k=0}^{j}(-1)^k\binom{2m}{m-k}&\binom{m+k-1}{k}\binom{m-k}{j-k}
\\
&= \sum_{k=0}^{j}(-1)^k \frac{(2m)!}{(m-k)!(m+k)!}\frac{(m+k-1)!}{k!(m-1)!}\frac{(m-k)!}{(j-k)!(m-j)!}
\\
&
=\frac{(2m)!}{(m-j)!(m+j)!}\sum_{k=0}^{j}(-1)^k\frac{(m+k-1)!}{k!(m-1)!}\frac{(m+j)!}{(j-k)!(m+k)!}
\\
&
=\binom{2m}{m+j}\sum_{k=0}^{j}(-1)^k\binom{m+k-1}{k}\binom{m+j}{m+k} =\binom{2m}{m+j},
\end{align*}
where the last equation follows from \eqref{binom_id}. Therefore, we finally get
\begin{equation*}
\mathrm{Res}_{x}(f) = \frac{1}{(x-\bar{x})^{2m}}\sum_{j=0}^{m}\binom{2m}{m+j}x^{m+j}\bar{x}^{m-j} = \frac{1}{(x-\bar{x})^{2m}}\sum_{j=m}^{2m}\binom{2m}{j}x^{j}\bar{x}^{2m-j}.
\end{equation*}
All the same, we have
\begin{align*}
\mathrm{Res}_{\bar{x}}(f) & = \frac{1}{(m-1)!}\lim_{s\to\bar{x}}\frac{\partial^{m-1}}{\partial s^{m-1}}\left((s-\bar{x})^m f(s)\right)
\\
&
=\frac{1}{(m-1)!}\lim_{s\to\bar{x}}\frac{\partial^{m-1}}{\partial s^{m-1}}\left(\frac{1}{(s-x)^{m+1}}s^{2m}\right)
\\
&
=\frac{1}{(m-1)!}\lim_{s\to\bar{x}}\sum_{k=0}^{m-1}\binom{m-1}{k}\left(\frac{\partial^k}{\partial s^{k}}\frac{1}{(s-x)^{m+1}}\right)\left(\frac{\partial^{m-1-k}}{\partial s^{m-1-k}}s^{2m}\right)
\end{align*}
and also
\begin{align*}
\mathrm{Res}_{\bar{x}}(f) &
=\frac{1}{(m-1)!}\lim_{s\to\bar{x}}
\sum_{k=0}^{m-1}\binom{m-1}{k}(-1)^k\frac{(m+k)!}{m!}\frac{1}{(s-x)^{m+1+k}}\frac{(2m)!}{(m+k+1)!}s^{m+k+1}
\\
&
= \sum_{k=0}^{m-1}(-1)^k \binom{2m}{m-k-1}\binom{m+k}{k} \frac{\bar{x}^{m+k+1}}{(\bar{x}-x)^{m+1+k}}
\\
&
=\frac{1}{(x-\bar{x})^{2m}} \sum_{k=0}^{m-1}(-1)^k \binom{2m}{m-k-1}\binom{m+k}{k}(-\bar{x})^{m+k+1}(x-\bar{x})^{m-k-1}.
\end{align*}
As we have
\begin{equation*}
(x-\bar{x})^{m-k-1} = \sum_{j=0}^{m-k-1}\binom{m-k-1}{j}x^j(-\bar{x})^{m-k-1-j},
\end{equation*}
this equals
\begin{align*}
\mathrm{Res}_{\bar{x}}(f) &= \frac{1}{(x-\bar{x})^{2m}} \sum_{k=0}^{m-1}\sum_{j=0}^{m-k-1}(-1)^k \binom{2m}{m-k-1}\binom{m+k}{k}\binom{m-k-1}{j}x^j(-\bar{x})^{2m-j}
\\
&
= \frac{1}{(x-\bar{x})^{2m}} \sum_{j=0}^{m-1}\sum_{k=0}^{m-j-1}(-1)^k \binom{2m}{m-k-1}\binom{m+k}{k}\binom{m-k-1}{j}x^j(-\bar{x})^{2m-j}.
\end{align*}
For the coefficients, we again obtain
\begin{align*}
\sum_{k=0}^{m-j-1}(-1)^k &\binom{2m}{m-k-1}\binom{m+k}{k}\binom{m-k-1}{j}
\\
&
= \sum_{k=0}^{m-j-1}(-1)^k \frac{(2m)!}{(m-k-1)!(m+k+1)!}\frac{(m+k)!}{k!m!}\frac{(m-k-1)!}{j!(m-k-j-1)!}
\\
&
=\frac{2m!}{j!(2m-j)!} \sum_{k=0}^{m-j-1}(-1)^k \frac{(2m-j)!}{(m+k+1)!(m-k-j-1)!}\frac{(m+k)!}{k!m!}
\\
&
=\binom{2m}{j}\sum_{k=0}^{m-j-1}(-1)^k\binom{2m-j}{m+k+1}\binom{m+k}{k} = \binom{2m}{j},
\end{align*}
where the last equation again follows from \eqref{binom_id}.
Thus, we finally obtain
\begin{equation*}
\mathrm{Res}_{\bar{x}}(f) = \frac{1}{(x-\bar{x})^{2m}}\sum_{j=0}^{m-1}\binom{2m}{j}x^j(-\bar{x})^{2m-j}.
\end{equation*}
Putting all this together, we get
\begin{align*}
\mathcal{P}_{n-1,n}(x) &= \gamma_n\left(\mathrm{Res}_x(f) + \mathrm{Res}_{\bar{x}}(f)\right)
\\
&
=\gamma_n\left(\frac{1}{(x-\bar{x})^{2m}}\sum_{j=m}^{2m}\binom{2m}{j}x^{j}\bar{x}^{2m-j} + \frac{1}{(x-\bar{x})^{2m}}\sum_{j=0}^{m-1}\binom{2m}{j}x^j(-\bar{x})^{2m-j}\right)
\\
&
=\frac{\gamma_n}{(x-\bar{x})^{2m}}\sum_{j=0}^{2m}\binom{2m}{j}x^{j}\bar{x}^{2m-j}=
\frac{\gamma_n}{(x-\bar{x})^{2m}}(x-\bar{x})^{2m}=\gamma_n.
\end{align*}
\end{proof}

\end{document}